\documentclass[12pt,leqno]{article}

\usepackage{amsfonts,amsthm,amssymb,amsmath,graphicx}

\newcommand{\fg}{\mathfrak {g}}
\newcommand{\fh}{\mathfrak {h}}

\newcommand{\fl}{\mathfrak {l}}
\newcommand{\fp}{\mathfrak {p}}

\newcommand{\fo}{\mathfrak {o}}
\newcommand{\fs}{\mathfrak {s}}

\newcommand{\gz}{\mathfrak {g}_{\bar{0}}}

\newcommand{\si}{\sigma}

\newcommand{\sk}{\bigskip}

\newcommand{\dy}{{\rm D}}

\newcommand{\al}{\alpha}

\newcommand{\ep}{\varepsilon}

\newcommand{\aut}{{\rm Aut}}
\newcommand{\inv}{{\rm Inv}}
\newcommand{\inn}{{\rm Int}}

\newcommand{\Lz}{L_{\bar{0}}}
\newcommand{\Lo}{L_{\bar{1}}}

\newcommand{\ad}{{\rm ad}}

\newtheorem{proposition}{Proposition}[section]
\newtheorem{theorem}[proposition]{Theorem}
\newtheorem{corollary}[proposition]{Corollary}
\newtheorem{lemma}[proposition]{Lemma}

\newtheorem{example}[proposition]{Example}
\newtheorem{remark}[proposition]{Remark}

\newcommand{\bz}{\mathbb {Z}}
\newcommand{\bc}{\mathbb {C}}

\newcommand{\bn}{\mathbb {N}}
\newcommand{\br}{\mathbb {R}}

\newcommand{\fsp}{{\mathfrak s}{\mathfrak p}}

\newcommand{\id}{{\rm id}}

\usepackage{lscape}
\usepackage{color}
\usepackage{graphicx}

\newcommand{\fgaa}{\mathfrak{g}(1)}
\newcommand{\fgbb}{\mathfrak{g}(i)}
\newcommand{\fgcc}{\mathfrak{g}(-1)}
\newcommand{\fgdd}{\mathfrak{g}(-i)}

 \newcommand{\fgia}{\fgaa_A}
\newcommand{\fgib}{\fgaa_B}
\newcommand{\cU}{\mathcal {U}}

\usepackage{multirow}
 \usepackage{longtable}

 \usepackage[figuresright]{rotating}
 
 \usepackage{float}
 \usepackage{dynkin-diagrams}

     \newcommand{\kapg}{\kappa_{\fg}}

 \newcommand{\kapl}{\kappa_L} 
     
 \newcommand{\Eso}{E_s^I}
  \newcommand{\Est}{E_s^{I\!I}}

      \newcommand{\KKs}{\mathrm{K_s}}

 \usepackage{hyperref}   

\begin{document}

{\large \bf Generalized symmetric pairs, Lie superalgebras,  Dirac operators and Kostant constants }

\sk

Mingjing Zhang

School of Mathematics, 
China University of Mining and Technology,

Xuzhou 221116, China.

zhangmj@cumt.edu.cn

\sk

Fuhai Zhu

School of Mathematics, Nanjing University,

Nanjing 210093, China.

zhufuhai@nju.edu.cn

\sk

\noindent {\bf Abstract}: Lie superalgebras have conventionally been regarded as extensions of Lie algebras.
In this paper, we present a geometric construction of contragredient Lie superalgebras, with the exceptions of  $A(m, n)~(m, n>0)$ and $G(3)$. 

Let $\fg$ be a complex simple Lie algebra equipped with an order 4 automorphism $\sigma$. 
With respect to $\sigma$, there is the eigenspace decomposition $\fg=\fgaa+\fgbb+\fgcc+\fgdd$. 
When $\fgbb$ admits a $\fgaa$-invariant symplectic structure,
$(\fg, \si)$ is called  a \emph{generalized symmetric pair  with symplectic subspace}.
In this paper, we classify all such pairs and construct a natural Lie superalgebra structure on $L=\fgaa+\fgbb$ with $\Lz=\fgaa$ and $\Lo=\fgbb$.

Furthermore, we investigate the constants arising from the squares of two distinct Dirac operators: 
the Dirac operator $D$ for   $\mathfrak{g} = \mathfrak{g}^{\sigma^2} + \mathfrak{g}^{-\sigma^2}$, 
and the Dirac operator $D_s$ for  $L = \fgaa + \fgbb$. 
We prove that the constant $C$ in $D^2$ admits a remarkable factorization
\[
C=J^*(\Omega_{\Lz})\cdot \dim \Lo \cdot h^\vee_{\mathrm{sign}}, 
\] 
where $h^\vee_{\mathrm{sign}}$ is the signed dual Coxeter number of $L$ and $J^*(\Omega_{\Lz})$ is a mapped Casimir constant of its even part. This formula reveals an unexpected unification of symmetric space geometry and Lie superalgebra invariants.

Finally, we prove that the Kostant constant $\KKs = \left.\Omega_{\Lz}\right|_{\Lo}=\frac{1}{2}$ is universal for all contragredient Lie superalgebras with non-degenerate Killing form, 
which implies that the constant $C_s$ in $D_s^2$ satisfies $C_s=\frac{1}{16}\dim \Lo$. We then establish a bijection between such superalgebras and solutions of the superadmissibility index equations $\Est =\frac{1}{2},  \Eso = \frac{1}{2}$,  thereby recovering the classification from an invariant-theoretic perspective.

\sk


\sk


\sk


\newpage

\section{Introduction} \label{intro}
\setcounter{equation}{0}

According to \'{E}. Cartan, a semisimple Riemannian symmetric space is a
Riemannian manifold of the form $G/R$, where $G$ is a real
semisimple Lie group with Cartan involution $\si$ and
fixed point set $G^\si$, such that $\mbox{Lie}(G^\si) = \mbox{Lie}(R)$.
A pseudo-Riemannian symmetric space corresponds to
an involution $\si$ on $G$, and they are classified by Berger \cite{be}.
By ignoring the different covering spaces, this amounts to the pair $(\fg_\br, \fh)$, where $\fg_\br$ is a real semisimple Lie
algebra, and $\fh$  consists of the fixed points of an involution on $\fg_\br$. Such pairs $(\fg_\br, \fh)$ are known
as the (semisimple) locally symmetric pairs.
Furthermore, the symmetric pair is also attached to an ordered
pair of commuting involutions $(\theta, \sigma)$ on the complexification $\fg$ of $\fg_\br$
with $\theta$ a Cartan involution of $\fg_\br$.
Helminck \cite{Helm} classifies such pairs
$(\theta, \si)$, and furthermore distinguishes the inner and outer involutions. Huang \cite{Huang} also
classifies the locally symmetric pairs by square quadruplets. In \cite{ch}, Chuah and Huang use
double Vogan diagrams to provide another combinatorial classification
 of the locally symmetric pairs, or equivalently the commuting involutions on
complex semisimple Lie algebras.

A symplectic symmetric space is a symmetric space which is endowed with a 
symplectic structure invariant by the symmetries \cite{BCG}.
In the case of a semisimple symplectic symmetric space $M$, i.e. one
whose transvection group $G$ is semisimple, the action of $G$ on $M$ is certainly
strongly Hamiltonian (even if $M$ is not simply connected) and therefore such a
space yields an admissible coadjoint orbit of $G$ due to Kostant.
Therefore, symplectic symmetric spaces are interesting in representation theory.
Symplectic  symmetric spaces
are classified for the pseudo-K\"ahler cases \cite{sh}
and non-pseudo-K\"ahler cases \cite{bi}.
They exist if and only if $G^\si$ has a 1-dimensional center.
Algebraically, $(\fg, \si)$ with $\si\in \inv(\fg)$ is a symplectic pair if and only if $\fg^\si$ has a 1-dimensional center.

The study of symplectic symmetric spaces has been generalized in the following directions:
\begin{equation}
\begin{array}{ll}
    (a) & \mbox{symplectic $k$-symmetric spaces,  } \\
    \mbox{(b)}& \mbox{symplectic symmetric spaces for Lie superalgebras.}
  \end{array}
  \label{twod}
  \end{equation}

A homogeneous $k$-symmetric space is a homogeneous space generated by a Lie group automorphism
of order $k$. It is a natural generalization of the concept of a symmetric space. The
geometry and a classification of generalized symmetric spaces were developed by  Fedenko
and  Kowalski \cite{Fe, Ko}. In \cite{BT},    Boche\'nski and  Tralle
give a list of all symplectic 3-symmetric manifolds
with simple groups of transvections. They also show a method of constructing semisimple
(non-compact) symplectic $k$-symmetric spaces from a given (compact) K\"ahler $k$-symmetric
space.

In light of recent developments in supergroups and super symplectic geometry, together with applications in super representation theory  (see for instance \cite{FG, Tu}),
the study of symplectic symmetric space for Lie superalgebras becomes interesting.
The symmetric pairs, which infinitesimally are the characters of the symmetric spaces, are classified in \cite{Ser, CZ} in Lie superalgebra case.

In this paper, we start with the $4$-symmetric spaces of complex simple Lie algebras.
By ignoring the different covering spaces, this amounts to the pair $(\fg, \si)$, where $\fg$ is a complex simple Lie algebra,
and $\si$ is an order 4 automorphism of $\fg$.  That $\sigma\in \aut(\fg)$ has order $n$ means $\sigma^n=\id$ and $\sigma^{n-1}\neq \id$.
In the study of symplectic $4$-symmetric spaces (\ref{twod})(a), we have the decomposition $\fg=\fh\oplus \mathfrak{m}$ with $\fh=\fg^\si$,
and then consider the symplectic form on $\mathfrak{m}$.
Here we consider the  decomposition of $\fg$ into the sum of eigenspaces with respect to $\si$.
Denote this decomposition as \[
\fg=\fgaa+\fgbb+\fgcc+\fgdd,\]
where $\fg(a)$ is the eigenspace of $\si$ with eigenvalue $a\in \{\pm 1, \pm i\}$.
We consider the $\fgaa$-invariant symplectic form on $\fgbb$.
Note that $\fgaa+\fgcc=\fg^{\si^2}$ is the invariant subalgebra of the involution $\si^2$.
Then $\fgaa+\fgcc$ is naturally reductive. We find that the center of $\fgaa+\fgcc$ determines the  $\fgaa$-invariant symplectic form on $\fgbb$.

\begin{theorem}
Let $\fg=\fgaa+\fgbb+\fgcc+\fgdd$ be the decomposition with respect to the order 4 automorphism $\si$. Then $\fgbb$ admits a non-degenerate $\fgaa$-invariant symplectic form if and only if
$\fgaa+\fgcc$ has a 1-dimensional center  in $\fgcc$.
\label{main1}
\end{theorem}

For the pair $(\fg, \si)$, Theorem \ref{main1} provides a necessary and sufficient condition for $\fgbb$ to have a $\fgaa$-invariant symplectic form.
This can be regarded as an extension of symplectic symmetric pair.
We refer to it as a   \emph{generalized symmetric pair  with symplectic subspace}.

\begin{theorem}
All the pairs $(\fg, \sigma)$ satisfying Theorem \ref{main1} are listed in Table \ref{table1}.
\label{main2}

 \begin{center}
\begin{table}[ht]\renewcommand{\arraystretch}{2.2}
\centering
\begin{tabular}{|c|c|c|c|c| }
  \hline
& $\fg$ &  $\si$ & $\fgaa$ & $\fgaa+\fgcc$   \\
  \hline
(a) &{ $ A_{2n-1} $} &\includegraphics[width=3.5cm]{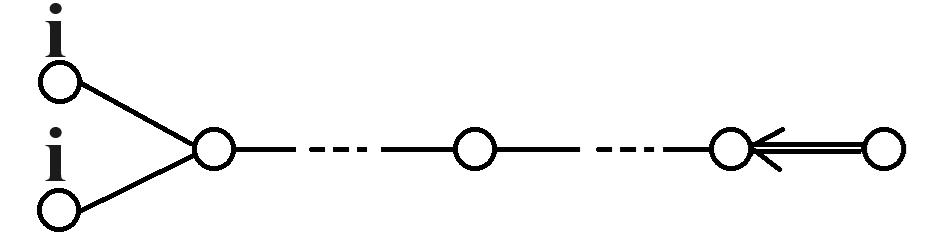}  &  {$ C_{n-1}\oplus \bc
  $} & {$A_{2n-3}\oplus A_1\oplus \bc$ }  \\
 \hline
(b)&{ $ A_{2n-1} $} &\includegraphics[width=3.5cm]{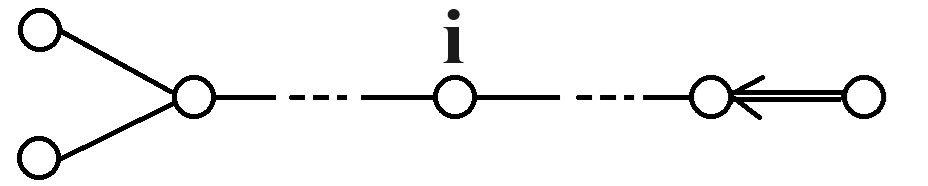}  &  { $ D_p\oplus C_{q}   $} & {$A_{2p-1}\oplus A_{2q-1}\oplus \bc$ }\\
 \hline
(c) &{ $ A_{2n} $}  &\includegraphics[width=3.5cm]{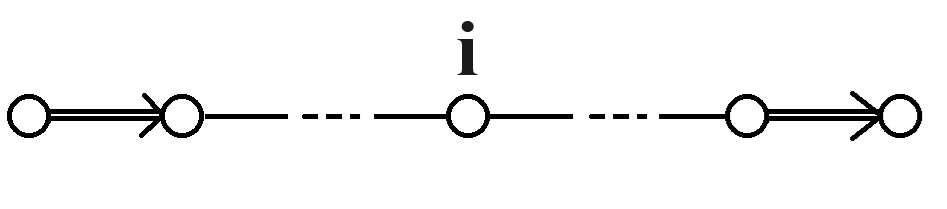}  &   { $ C_p\oplus B_{q}  $} & {$A_{2p-1}\oplus A_{2q}\oplus \bc$ }\\
 \hline
 (d)& { $ A_{2n} $}  &\includegraphics[width=3.5cm]{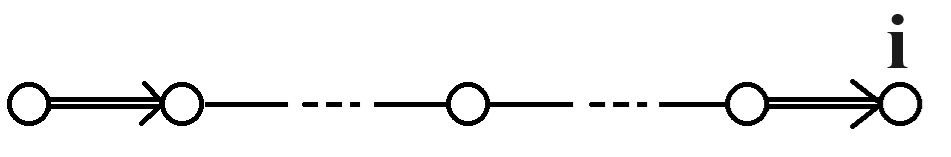}  &  { $ C_n
 $} & {$A_{2n-1} \oplus \bc$ } \\
  \hline
(e)& {$ D_{n}$}       &\includegraphics[width=3.5cm]{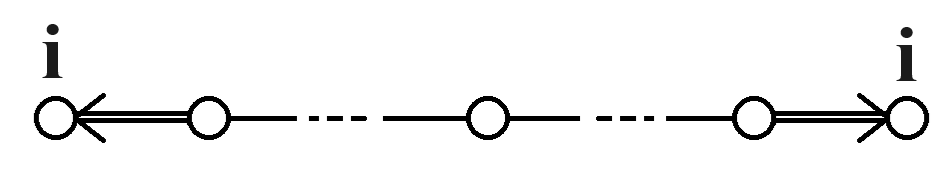}  &  { $  A_{n-2}\oplus  \bc  $} &   $  D_{n-1}\oplus \bc$  \\
  \hline
(f)  &  { $ E_6  $}    &\includegraphics[width=2.6cm]{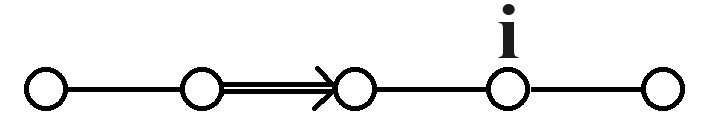}  &  { $B_{3}\oplus A_{1} $} &   $ D_{5}\oplus \bc$   \\
  \hline
\end{tabular}
\caption{Generalized symmetric pairs  with symplectic subspace}
\label{table1}
\end{table}
\end{center}
\end{theorem}

In Table \ref{table1}, $p+q=n$, and $\bc$ denotes the 1-dimensional center of the associated Lie subalgebra.
For example, in Table \ref{table1}(a), we have $\fgaa\cong C_{n-1}\oplus \bc$ and  $\fgaa+\fgcc\cong A_{2n-3}\oplus A_1\oplus \bc$.
The first $\bc$ denotes the 1-dimensional center of $\fgaa$, and meanwhile the second  $\bc$ denotes the 1-dimensional center of $\fgaa+\fgcc$.
By Theorem \ref{main1},  $\fgaa+\fgcc$ has  1-dimensional center in $\fgcc$, 
so the second $\bc$ lies in $\fgcc$. The two $\bc$'s carry different meanings.

Let $\aut(\fg)$ denote the group of automorphisms of $\fg$, and let $\inn(\fg)$ be its identity component.
An element in $\inn(\fg)$ is called inner, and otherwise it is called outer.
By Theorem \ref{main2}, we obtain the following straightforward corollary.
\begin{corollary}
 If $\sigma$ satisfies Theorem \ref{main1}, then $\sigma$ is outer.
\end{corollary}

With the classification in hand, we turn to the construction of Lie superalgebras. 
Let $L=\Lz+\Lo$ be a complex Lie superalgebra. If there exists a non-degenerate invariant bilinear form $B_L$ on $L$, which is symmetric on $\Lz$  and skew-symmetric on $\Lo$,
then $L$ is said to be of Riemannian type.
The simple Lie superalgebras of Riemannian type are referred to basic classical Lie superalgebras by V. Kac.
Besides ordinary simple Lie algebras, they are listed as $A(m,n)$, $B(m,n)$, $C(n)$, $D(m,n)$, $D(2,1;\alpha)~ (\alpha\neq 0, -1)$, $F(4)$ and $G(3)$ (\cite[\S 2.5]{ka77}).

Returning to the pair $(\fg, \si)$ with $\fg$ complex simple and $\si$ an order 4 automorphism of $\fg$, we have classified the generalized symmetric pairs  with symplectic subspace.
Moreover, since $\fg$ is simple, the Killing form of $\fg$ restricts to a non-singular symmetric bilinear form on $\fgaa$.
By comparing with the classification of basic classical Lie superalgebras, we observe that for each generalized symmetric pair with symplectic subspace $(\fg, \si)$
 in Table \ref{table1}, there exists a Lie superalgebra structure on $\fgaa+\fgbb$.
A natural question is whether we can construct the Lie superalgebra bracket from the symplectic structure on $\fgbb$ together with the Lie algebra structure of $\fg$.

Let $z\in \fgcc$ be a nonzero element in the center of $\fgaa+\fgcc$. There exists a  symmetric bilinear map:
\[
\fgbb\times \fgbb\longrightarrow \fgaa; ~~  (x, y)\longmapsto [[z, x], y].
\]
A natural question is whether this map yields the desired  Lie superalgebra bracket on $\fgaa+\fgbb$. 
When $\fg=A_{2n}$ and $\si$ is as in Table \ref{table1}(d) with $\fgaa\cong C_n$, this is indeed the case.
In general, however, this symmetric bilinear map does not induce a Lie  superalgebra structure on $\fgaa+\fgbb$. See Remark \ref{rem}.

We now define a map $J$ on $\fgaa$ as follows.
\begin{itemize}
\item \underline{Case 1: $(\fg, \si)$ in Table \ref{table1}(b) with $\fg=A_5$.}

In this case, $\fgaa=A_1^3=\sum_{j=1}^3 \bc\{X_j, H_j, Y_j\}$, and define
\[
J=\left\{ \begin{array}{ll}
2 \alpha \cdot \id, &  \mbox{ on $A_1=\bc\{X_1, H_1, Y_1\}$}, \\
  \id, &  \mbox{ on $A_1=\bc\{X_2, H_2, Y_2\}$},\\
-2(1+\alpha) \cdot   \id, &  \mbox{ on $A_1=\bc\{X_3, H_3, Y_3\}$},  \end{array}
\right.
\]
with $\alpha\neq 0, -1$.
\item \underline{Case 2: $(\fg, \si)$ in Table \ref{table1}(e).}

In this case, $\fgaa=A_{n-2}+\bc$, and define
\[
J=\left\{\begin{array}{ll}
\id, & \mbox{ on }A_{n-2}, \\
-(n-2)\cdot \id, & \mbox{ on } \bc. \end{array}\right.
\]

\item \underline{Case 3: $(\fg, \si)$ in Table \ref{table1}(f).}

In this case, $\fgaa=B_3+A_1$, and define
\[
J=\left\{\begin{array}{ll}
\id, & \mbox{ on } B_3, \\
-3\cdot  \id, & \mbox{ on } A_1. \end{array}\right.
\]

\item  \underline{Case 4: all remaining cases in Table \ref{table1}.}

Here  $\fgaa=\fgia+ \fgib$ is a direct sum of two ideals.
For $(\fg, \si)$ in Table \ref{table1}(a), $\fgia$ is simple and $\fgib\cong \bc$ is 1-dimensional center of $\fgaa$.
For $(\fg, \si)$ in Table \ref{table1}(d), $\fgaa$ is simple and  $\fgib$  is regarded as redundant.
For all other cases, both $\fgia$ and $\fgib$ are simple.
Define
\[
J=\left\{ \begin{array}{ll}
\id & \mbox{ on } \fgia, \\
-\id & \mbox{ on } \fgib.
\end{array}\right.
\]
\end{itemize}

We state the following result.
 \begin{theorem}
Let $(\fg, \si)$ be a generalized symmetric pair with symplectic subspace, which is listed in Table \ref{table1}.
Then with respect to $\si$, we have the eigenspace decomposition
\[
\fg=\fgaa+\fgbb+\fgcc+\fgdd.
\]
Suppose $z\in \fgcc$ to be a nonzero element in the center of  $\fgaa+\fgcc$.
Let $L=\fgaa+\fgbb$ and let $J$ be defined as above.
Define a symmetric bilinear map $[\cdot, \cdot]_s: L \times L \rightarrow L$ by
\begin{equation}
[x, y]_s=\left\{ \begin{array}{ll}
[x, y], & \mbox{if $x\in \fgaa$ or $y\in \fgaa$}; \\
J([[z, x], y]), &\mbox{if $x, y\in \fgbb$}. \end{array}\right. 
\label{superbracket}
\end{equation}
Then this map defines a   Lie superalgebra structure on $L=\fgaa+\fgbb$ with  $\Lz=\fgaa$ and $\Lo=\fgbb$.

Moreover, the Lie superalgebras obtained via this method include
\[
 B(m,n), C(n), D(m,n), A(m,0), B(0, n), F(4),  D(2,1;\alpha) ~(\alpha \neq 0, -1).
\]
\label{main}
\end{theorem}

In the traditional study of Lie theory, Lie superalgebras have conventionally been regarded as extensions of Lie algebras.
Consequently, the investigation of Lie superalgebras has typically relied on the theory of Lie algebras and their representation theory.
In this paper, we construct Lie superalgebras from generalized symmetric pairs with symplectic subspace, bypassing the need for traditional root system methods.
This approach can be viewed as a geometric realization of classical Lie superalgebras, and it yields almost all classical Lie superalgebras with the exception of $G(3)$.
For $A(m, n)$ with $m, n>0$, as well as a discussion of  $G(3)$, see Remark \ref{remAmn}.

For basic Lie superalgebras $A(m, n)$, $B(m, n)$, $C(n)$, $D(m,n)$ and $B(0,n)$,
explicit matrix realizations of their elements can be readily constructed.
By contrast,  providing a unified and explicit geometric construction of exceptional Lie superalgebras remains an interesting challenge.
The exceptional Lie superalgebras $F(4)$, $G(3)$ and $D(2,1;\alpha)$ ($\alpha \neq 0, -1$) were first discovered by Freund and Kaplansky \cite{FK},
with their explicit constructions later provided by Scheunert, Nahm, and Rittenberg \cite{SNR}.
Specifically, the exceptional Lie superalgebra $F(4)$ is typically constructed via Clifford algebra techniques.
In this paper, we present an alternative geometric realization of $F(4)$ and $D(2,1;\alpha)$ directly from the $4$-symmetric spaces of $E_6$ and $A_5$ respectively, 
offering new insights into their structural origin and potential applications in mathematical physics.

Having established the geometric construction, we next investigate  the constants arising from the squares of two distinct Dirac operators: 
the Dirac operator $D$ for the Lie algebra $\mathfrak{g} = \mathfrak{g}^{\sigma^2} +\mathfrak{g}^{-\sigma^2}$, 
and the Dirac operator $D_s$ for the associated Lie superalgebra $L = \fgaa + \fgbb$.  
In the 1970s, Dirac operators were introduced into representation theory by Parthasarathy \cite{Par} 
as a tool to construct the discrete series representations. This approach was further developed by Atiyah and Schmid \cite{AS}, 
who constructed all discrete series representations as kernels of the Dirac operator.
In the 1990s, Vogan defined the Dirac operator in a completely algebraic setting. He introduced the concept of Dirac cohomology and made a conjecture on the infinitesimal character of irreducible representations with nonzero Dirac cohomology. This conjecture was proved by Huang and Pand\v{z}i\'c \cite{HP1}.
Kostant generalized this result to the case of his cubic Dirac operator and applied it to the topology of homogeneous spaces \cite{Ko03}.
The Dirac operator was also defined for a Lie superalgebra of Riemannian type by Huang and Pand\v{z}i\'c, and an analog of Vogan's conjecture for this Dirac operator was proved in \cite{HP05}. 

Let $(\fg, \si)$ be a generalized symmetric pair with symplectic subspace, and let $L=\fgaa+ \fgbb$ be the Lie superalgebra associated to this pair in the sense of Theorem \ref{main}.
Then $\sigma^2$  is evidently an involution of Hermitian type of $\mathfrak{g}$. 
Let $\kapg(\cdot, \cdot)$ be the Killing form of $\fg$. 
Choose $z\in \fgcc$  in the center of $\fgaa+\fgcc$ and define $\langle \cdot, \cdot \rangle: \fgbb\times \fgbb \rightarrow \bc$
  by
\[
\langle X, Y\rangle =\kapg([z, X], Y), ~ \forall X, Y\in\fgbb. 
\]
This gives a symplectic form on the odd part of $L$.  

Choose $\{x_1, x_2, \ldots, x_s, \partial_1, \partial_2, \ldots, \partial_s\}$ to be a basis of $\fgbb$ satisfying $\langle x_i, \partial_j\rangle=\delta_{ij}$
and $\langle x_i , x_j\rangle=\langle \partial_i ,\partial_j\rangle=0$ for $1\leq i, j\leq s$.
Then $\{-[z, \partial_1], -[z, \partial_2], \ldots, -[z, \partial_s], [z, x_1],   $ \\$ [z, x_2], \ldots, [z, x_s]\}$ is an orthonormal basis of $\fgdd$ with respect
to the symplectic form $\langle \cdot, \cdot \rangle= \kapg([z, \cdot], \cdot)$.

The Dirac operator for the Lie algebra $\fg=\fg^{\si^2}\oplus \fg^{-\si^2}$ is
\begin{eqnarray*}
D&=&\sum_{i=1}^s (-x_i\otimes [z, \partial_i]+ \partial_i \otimes [z, x_i]- [z, \partial_i]\otimes x_i + [z, x_i]\otimes\partial_i)  \\
&& \in \mathcal{U}(\fg)\otimes \mathcal{C}(\fgbb+\fgdd).
\end{eqnarray*}
where $\mathcal{C}(\fgbb+\fgdd)$ denotes the Weyl algebra of $\fg^{-\si^2}=\fgbb+\fgdd$ with respect to the Killing form $\kapg(\cdot, \cdot)$ on it. 
Meanwhile, the Dirac operator for the Lie superalgebra $L=\fgaa+\fgbb$ is
\[
D_s=\sum_{i=1}^s (x_i\otimes \partial_i-\partial_i\otimes x_i) \in \mathcal{U}(L)\otimes \mathcal{W}(\fgbb),
\]
where $ \mathcal{W}(\fgbb)$ denotes the Weyl algebra of $\fgbb$ with respect to the symplectic form $\langle \cdot, \cdot\rangle$ on it.

Let $C$  and $C_s$ be the constants in $D^2$ and $D_s^2$, respectively. 
Suppose that $L$ has a non-degenerate Killing form.  
Let $\{X_i\}$ be a basis of $\Lz=\fgaa$, and let $\{X_i^*\}$ be its dual basis with respect to the Killing form.
The Casimir element of $\Lz$ is defined by $\Omega_{\Lz}=\sum X_i\otimes X_i^*$.
Now we define
\begin{equation}
J^*(\Omega_{\Lz})=\sum X_i\otimes J^{-1}(X_i^*).
\end{equation}
For simplicity, we also denote the constant of $J^*(\Omega_{\Lz})$ on $\Lo$ by  $J^*(\Omega_{\Lz})$.
Then we have the following formula. 
\begin{theorem}
Let $(\fg, \si)$ be a generalized symmetric pair with symplectic
subspace, and let $L=\fgaa + \fgbb$ be the associated Lie superalgebra as in Theorem \ref{main}.
Assume that $L$ has a non-degenerate Killing form. Then we have
\[
C=J^*(\Omega_{\Lz})\cdot \dim \Lo\cdot h^\vee_{\mathrm{sign}},
\]
where $C$ is the constant in $D^2$ for $\fg=\fg^{\si^2}+ \fg^{-\si^2}$, and $h^\vee_{\mathrm{sign}}$ is the signed dual Coxeter number of $L$ given in Table \ref{tab:dualcoxeter}.
\label{mainformula}
\end{theorem}

 This formula reveals an unexpected unification: the constant $C$, which arises purely from the geometry of the $4$-symmetric space $(\fg,\sigma)$, decomposes into a product of three invariants of the associated Lie superalgebra $L$.  
 This can be seen as evidence of the deep interplay between symmetric space geometry and Lie superalgebra theory.

Next, we prove a general result on the Dirac constant for contragredient Lie superalgebras. 
In \cite{Kos1} and \cite{Kos}, B. Kostant studies $\Lz+ \Lo$, where $\Lz$ is a Lie algebra which admits a non-singular symmetric bilinear form,
and $\Lo$ has a non-singular $\Lz$-invariant alternating bilinear form. 
He defines a map in (0.4) and (0.5) of \cite{Kos}: 
\[\nu^*: \Lz \hookrightarrow \fsp(\Lo)\cong S^2(\Lo) \hookrightarrow \mathcal{W}(\Lo).\]
He also gives a necessary and sufficient condition for $\Lz\oplus \Lo$ to become a Lie superalgebra, namely that $\nu^*(\Omega_{\Lz})$ is a constant in $\mathcal{W}(\Lo) $ for  the Casimir element $\Omega_{\Lz}$ of $\Lz$.  
\begin{theorem}
Let $L$ be a complex contragredient Lie superalgebra with non-degenerate Killing form. Then the Casimir element $\Omega_{\Lz}$ of 
$\Lz$ acts on $\Lo$ as a constant, which is denoted by $\KKs$. Moreover, we have 
\[\KKs=\left. \Omega_{\Lz}\right|_{\Lo}=\displaystyle\frac{1}{2}. \] 
\label{main12}
\end{theorem}
For simplicity,  we call $\KKs$ the \emph{Kostant constant} for Lie superalgebra. 
Then this implies naturally the following corollary.

\begin{corollary}
Let $L$ be a complex contragredient Lie superalgebra with nondegenerate Killing form.
Then the constant in $D_s^2$ satisfies 
\[
C_s=\nu^*(\Omega_{\Lz})=\frac{1}{16} \dim \Lo.
\] 
\label{KCos}
\end{corollary}

Let $\mathfrak{g}$ and $\mathfrak{h}$ be complex simple Lie algebras with irreducible representations $U$ and $V$, respectively.
Based on the study of the constant $\KKs$, it is reasonable to define  the  \emph{superadmissibility index of type $I\!I$} for $(\fg, U)$ and $(\fh, V)$ as  
\[
\Est(\fg, U; \fh, V) = \frac{C_U}{1 - l_U \dim V} + \frac{C_V}{1 - l_V \dim U}, 
\] 
and its degenerate form (when $\fh=0$ and $V$ is trivial)  as
 \[
\Est(\fg, U ) = \frac{C_U}{1 - l_U }.
\]

Similarly, we define the \emph{superadmissibility index of type $I$} for $(\fg, U)$ and $(\fh, V)$ as:    
\[
\Eso(\fg, U; \fh, V) = \frac{C_U}{1 - 2l_U \dim V} + \frac{C_V}{1 - 2l_V \dim U}-\frac{1}{2\dim U \dim V},
\]
and its degenerate form as 
\[
\Eso(\fg, U) = \frac{C_U}{1 - 2l_U  }  -\frac{1}{2\dim U}.
\]
 
Finally, we study the solutions of the superadmissibility index equations $\Est=\frac{1}{2}$ and $\Eso=\frac{1}{2}$, and hence obtain the following result. 
\begin{theorem}
\noindent{(1)} There is a one-to-one correspondence between type 2 Lie superalgebra $L$ with non-degenerate Killing form and the pairs $((\fg, U), (\fh, V))$ with $\Est(\fg, U; \fh, V) =\displaystyle \frac12$. 
The correspondence is determined by $\Lz\cong \fg\oplus \fh$ and $\Lo\cong U\otimes V$. 

\noindent{(2)}  There is a one-to-one correspondence between type 1 Lie superalgebra $L$ with non-degenerate Killing form and the pairs $((\fg, U), (\fh, V))$ with $\Eso(\fg, U; \fh, V) =\displaystyle \frac12$. 
The correspondence is determined by $\Lz\cong \fg\oplus \fh\oplus\bc$ and $\Lo\cong U\otimes V\oplus U^*\otimes V^*$. 
\label{main4}
\end{theorem}

Theorems \ref{main12} and   \ref{main4}  constitute the super analogue of Panyushev's result on involutions of Lie algebras \cite{Pan}, but with an important new feature: the constant $\KKs = 1/2$ is universal both for all contragredient Lie superalgebras with non-degenerate Killing form and for Lie algebra decompositions with respect to involutions. Furthermore,  the superadmissibility index equations $\Est=\frac{1}{2}$ and $\Eso = \frac{1}{2}$  completely characterize the classification of Lie superalgebras. 
This provides an invariant-theoretic reconstruction of the classification of contragredient Lie superalgebras with non-degenerate Killing form, recovering the complete list from a single numerical condition--the superadmissibility index attaining the value $\frac{1}{2}$.

The remainder of this article is organized as follows. Section 2 recalls preliminary results that will be used throughout the paper. In Section 3, we prove Theorems \ref{main1} and \ref{main2}. In Section 4, we prove Theorem \ref{main}. Section 5 is devoted to the proof of Theorem \ref{mainformula}. Finally, in Section 6, we prove Theorem \ref{main12} and Corollary \ref{KCos}, and also establish Theorem \ref{main4} concerning the superadmissibility index.  

\bigskip

\noindent {\bf Acknowledgements. }
M. Zhang is supported by the
National Natural Science Foundation of China (Grant No. 11901588).
F. Zhu is supported by the National Natural Science Foundation of China (Grant No. 11571182).
The authors thank Meng-Kiat Chuah and Huibin Chen for very helpful suggestions.


\section{Preliminary}

In this section, we will recall some preliminary results, which will be used in the following sections.

\subsubsection*{Dynkin index}
Let $\fg$ be a complex simple Lie algebra. Let $(\rho, U)$ be an irreducible representation of $\fg$. Let $\kappa_\fg$ be the  Killing form of $\fg$. 
The \emph{Dynkin index} of $U$ is defined by
\[
\mathrm{tr}_U(\rho(X)\rho(Y)) = l_U \cdot \kappa_\fg(X, Y).
\]
The index of a direct sum of representations is the sum of the indices
of the representations.
 
The index of a one-dimensional representation of a simple Lie algebra $\mathfrak{g}$
is $0$, and that for irreducible representation is $1$ only of the adjoint representation.
The  Dynkin indices with $0<l_U<1$ for simple Lie algebras can be found in \cite{AVE} or \cite[Table III]{ka77}. For the readers' convenience, we also provide them  in Table \ref{Index}. 

\begin{table}[ht] 
\begin{tabular}{lllclll}
\hline
Type & & $r$ & Diagram & $\dim U$ & $l_U$ &  $C_U$ \\
\hline
$A$ & $\mathfrak{sl}_n$ & $n-1$ & $\dynkin[mark=o, labels={1}]{A}{oo.oo}$ & $n$ & $\frac{1}{2n}$ & $\frac{n^2-1}{2n^2}$ \\[6pt]
 & $\Lambda^2\mathfrak{sl}_n$ & $n-1$ &  $\dynkin[mark=o, labels={,1}]{A}{oo.oo}$ & $\frac{n(n-1)}{2}$ & $\frac{n-2}{2n}$ & $\frac{(n-2)(n+1)}{n^2}$ \\[6pt]
 & $S^2\mathfrak{sl}_n$ & $n-1$ &  $\dynkin[mark=o, labels={2}]{A}{oo.oo}$ & $\frac{n(n+1)}{2}$ & $\frac{n+2}{2n}$& $\frac{(n-1)(n+2)}{n^2}$ \\[6pt]
 & $\Lambda^5\mathfrak{sl}_n,\ n=6,7,8$ & $5,6,7$ &  $\dynkin[mark=o, labels={,,1}]{A}{ooo.o}$& $20,35,56$ & $\frac{1}{2},\frac{5}{7},\frac{15}{16} $ & 
$\frac{7}{8}, \frac{48}{49}, \frac{135}{128}$ \\[6pt]
\hline
$B$ & $\mathfrak{so}_n$ & $\frac{n-1}{2}$ &   $\dynkin[mark=o, labels={ 1}]{B}{oo.oo}$ & $n$ & $\frac{1}{n-2}$ & $\frac{n-1}{2(n-2)}$\\[6pt]
 & $\mathrm{spin}_n,\ n=7,9,11,13$ & $3,4,5,6$ &  $\dynkin[mark=o, labels={,,,1}]{B}{oo.oo}$ & $8,16,32,64$ & $\frac{1}{5}, \frac{2}{7},\frac{4}{9},\frac{8}{11}$  & $\frac{21}{40}, \frac{9}{14}, \frac{55}{72}, \frac{39}{44}$ \\[6pt]
\hline
$C$ & $\mathfrak{sp}_n$ & $\frac{n}{2}$ &  $\dynkin[mark=o, labels={ 1}]{C}{oo.oo}$ & $n$ & $\frac{1}{n+2}$ & $\frac{n+1}{2(n+2)}$ \\[6pt]
 & $\Lambda_0^5\mathfrak{sp}_n$ & $\frac{n}{2}$ &  $\dynkin[mark=o, labels={ ,1}]{B}{oo.oo}$ & $\frac{n(n-1)}{2}-1$ & $\frac{n-2}{n+2}$ & $\frac{n}{n+2}$ \\[6pt]
 & $\Lambda_0^5\mathfrak{sp}_3$ & $3$ &  $\dynkin[mark=o, labels={,,1}]{C}{3}$ & $14$ & $\frac{5}{8}$ & $\frac{15}{16}$ \\[6pt]
\hline
$D$ & $\mathfrak{so}_n$ & $\frac{n}{2}$ & $\dynkin[mark=o, labels={ 1}]{D}{oo.ooo}$& $n$ & $\frac{1}{ n-2}$ & $\frac{n-1}{2(n-2)}$  \\[6pt]
 & $\mathrm{spin}_n,\ n=10,12,14$ & $5,6,7$ & $\dynkin[mark=o, labels={,,,, 1}]{D}{oo.ooo}$ & $16,32,64$ & $\frac{1}{4},\frac{2}{5},\frac{2}{3}$ & $\frac{45}{64},\frac{33}{40},\frac{91}{96}$\\[6pt]
\hline
$E$ & $E_6$ & $6$ & $\dynkin[mark=o, labels={1}]{E}{6}$ & $27$ & $\frac{1}{4}$ & $\frac{13}{18}$\\[6pt]
 & $E_7$ & $7$ & $\dynkin[mark=o, labels={,,,,,,1}]{E}{7}$ & $56$ & $\frac{1}{3}$ & $\frac{19}{24}$\\[6pt]
\hline 
$F$ & $F_4$ & $4$ & $\dynkin[mark=o, labels={1}]{F}{4}$ & $26$ & $\frac{1}{3}$& $\frac{2}{3}$ \\[6pt]
\hline
$G$ & $G_2$ & $2$ & $\dynkin[mark=o, labels={ 1}]{G}{2}$ & $7$ & $\frac{1}{4}$& $\frac{1}{2}$ \\[6pt]
\hline
\end{tabular}
\caption{Dynkin index less than 1 and Casimir constant}
\label{Index}
\end{table}

\subsubsection*{Coxeter number and dual Coxeter number} 

For a finite irreducible root system $\Phi$ associated to a complex simple Lie algebra $\mathfrak{g}$, 
the Coxeter number $h$ is defined by any of the following equivalent characterizations:
\begin{itemize}
\item[(1)] $h = \frac{|\Phi|}{r}$, where $|\Phi|$ is the total number of roots, and $r$ is the rank of the root system.
\item[(2)]  The order of the Coxeter element (product of all simple reflections in the Weyl group) of the root system.
\end{itemize}

The dual Coxeter number $h^\vee$  of $\mathfrak{g}$ is defined as 
\[
h^\vee =1+(\rho, \theta^\vee)=1 + \frac{1}{(\theta, \theta)}\sum_{\alpha \in \Phi^+} (\theta, \alpha).
\] 
For complex  simple Lie algebras, they can be found in \cite[Table I]{KW}.

The dual Coxeter number $h^\vee$  of a basic Lie superalgebra is defined by the generalization of the Freudenthal--de Vries strange formula:
\[
\langle\rho, \rho\rangle = \frac{h^\vee}{24} (\dim \Lz-\dim\Lo),
\]
where $\langle \cdot, \cdot\rangle $ is the normalized bilinear form on $\mathfrak{h}^*$ such that the longest even roots have squared length $2$.
Dual Coxeter numbers for Lie superalgebras are listed in Table \ref{tab:dualcoxeter} (\cite[Table 1]{Peng}).  
Note that the sign choice in the definition of the map $J$ (Theorem \ref{main}) is inconsequential. Hence, for the dual Coxeter number, we do not always insist on its positivity. 
A minus sign may be added, or even multiplied by $\frac{1}{2}$, thereby obtaining the definition of the signed dual Coxeter number  $h^\vee_{\rm{sign}}$, which is listed in 
the last column of Table \ref{tab:dualcoxeter}. 

\begin{table}[ht]
\centering
\begin{tabular}{lccc}
\hline
Lie superalgebra $\mathfrak{g}$ & Standard notation &   $h^\vee$ & $h^\vee_{\rm{sign}}$ \\
\hline
$A(m,n)$ & $\mathfrak{sl}(m+1,n+1)$ &$  | m-n | $ & $m-n$  \\  
$B(m,n)$ & $\mathfrak{osp}(2m+1,2n)$ & {$\left\{\begin{array}{ll} 2(m-n)-1 & m>n \\ 
n-m+\frac{1}{2} & m\leq n\end{array}\right. 
$} & $ m-n-\frac{1}{2}$ \\
$B(0,n)$ & $\mathfrak{osp}(2,2n)$ & $\frac{2n+1}{2}$ & $\frac{2n+1}{2}$ \\
$C(n)$ & $\mathfrak{osp}(2,2n-2)$ &  $n-1$ & $n-1$ \\
$D(m,n)$ & $\mathfrak{osp}(2m,2n)$ &  {$\left\{\begin{array}{ll} 2(m-n-1) & m>n \\ 
n-m+1 & m \leq n\end{array}\right. 
$} & $ m-n-1 $ \\
$D(2,1;\alpha)$ & exceptional family & $0$& $0$ \\
$F(4)$ & exceptional & $3$ & $3$ \\
$G(3)$ & exceptional & $2$  & $2$ \\
\hline
\end{tabular}
\caption{Dual Coxeter numbers of  Lie superalgebras }
\label{tab:dualcoxeter}
\end{table}

\subsection*{Fundamental weight}
Let $\Delta = \{\alpha_1,\dots,\alpha_r\}$ be a  simple roots of $\fg$.
The \emph{fundamental weights}
$\omega_1,\dots, \omega_r \in \mathfrak{h}^*$ are defined by the duality
condition
\[
( \omega_i,\, \alpha_j^\vee ) = \delta_{ij},
\qquad
\alpha_j^\vee = \frac{2\alpha_j}{(\alpha_j,\alpha_j)},
\qquad
1\le i,j \le r,
\]
where $(\cdot, \cdot)$ is the inner product  on $\mathfrak{h}^*$ induced by the Killing form, and $\delta_{ij}$ is the Kronecker delta.  
Every finite-dimensional irreducible representation of $\mathfrak{g}$ is
uniquely determined by its highest weight 
$\lambda = \sum_{i=1}^r m_i \omega_i$, where each $m_i$ is a
non-negative integer.   
The information about the fundamental weights for complex  simple Lie algebras can be found in 
\cite[Table 1]{Hum}. 

If $\fg$ is a classical Lie algebra, its standard representation $V$ has highest weight $\omega_1$, the first fundamental weight. 
We also represent $V_{\omega_1}$ by $\bc^{\dim V}$.
For example, when $\fg=B_{m}$, we write its standard representation $V_{\omega_1}$ as $\bc^{2m+1}$.
Similarly, we write $\bc^{2n}$  for the standard representation of $D_n$.  
The reason is that when we sometimes consider the semisimple Lie algebra such as $\fg=B_m\oplus D_n$, the tensor product of the standard representations of $B_m$  and $D_n$  yields an irreducible representation of $\fg$. In this context, $\bc^{2m+1} \otimes \bc^{2n}$  is a more convenient notation than $V_{\omega_1} \otimes V_{\omega_1}$.

\subsubsection*{Weyl vector and its norm}

For a finite-dimensional simple complex Lie algebra $\mathfrak{g}$ with root system $\Delta$, choose a set of positive roots $\Delta^+$. The \emph{Weyl vector} (or half-sum of positive roots) is defined as 
$\rho = \frac{1}{2} \sum_{\alpha \in \Delta^+} \alpha$.
An important formula for calculating the norm of $\rho$ is the \emph{Freudenthal--de Vries strange formula}:
\begin{theorem}[Freudenthal--de Vries]
Let $h^\vee$ be the dual Coxeter number of $\mathfrak{g}$ and $\dim(\mathfrak{g})$ its dimension. Under the normalization where the squared length of long roots is $2$, we have:
$$
\|\rho\|^2 = \frac{1}{24} h^\vee \dim(\mathfrak{g})
$$
\end{theorem} 
The following table presents the Weyl vector and its norm for all irreducible root systems, with $\omega_i$ denoting the fundamental dominant weights:
 
\begin{table}[ht] {\footnotesize 
\centering
\renewcommand{\arraystretch}{1.5}
\setlength{\tabcolsep}{6pt}
\begin{tabular}{lccccc}
\hline
Type & Rank & dimension &Coxeter number& $\rho$ (fundamental weight basis) & $\|\rho\|^2$  \\
\hline
$A_n$   & $n$ & $n(n+2)$ & $n+1$ & $\omega_1 + \omega_2 + \cdots + \omega_n$ & $\frac{n(n+1)(n+2)}{12}$ \\
$B_n$  & $n$ & $n(2n+1)$ & $2n-1$ &  $\omega_1 + 2\omega_2 + \cdots + n\omega_n$ & $\frac{n(4n^2 - 1)}{12}$ \\
$C_n$   & $n$ & $n(2n+1)$ &  $n+1$ & $\omega_1 + 2\omega_2 + \cdots + n\omega_n$ & $\frac{n(n+1)(2n+1)}{6}$ \\
$D_n$   & $n$ & $n(2n-1)$ & $2n-2$& $\omega_1 + 2\omega_2 + \cdots + (n-2)\omega_{n-2} + \frac{n}{2}\omega_{n-1} + \frac{n}{2}\omega_n$ & $\frac{n(n-1)(2n-1)}{6}$ \\
$E_6$ & 6 & 78 &12&  $\omega_1 + 2\omega_2 + 3\omega_3 + 2\omega_4 + \omega_5 + 2\omega_6$ & 78  \\
$E_7$ & 7 &133 & 18&$2\omega_1 + 3\omega_2 + 4\omega_3 + 3\omega_4 + 2\omega_5 + \omega_6 + 2\omega_7$ & 133 \\
$E_8$ & 8 & 248 &30& $3\omega_1 + 4\omega_2 + 5\omega_3 + 6\omega_4 + 4\omega_5 + 2\omega_6 + 3\omega_7 + 2\omega_8$ & 248  \\
$F_4$ & 4 & 52 &9& $2\omega_1 + 3\omega_2 + 4\omega_3 + 2\omega_4$ & 39 \\
$G_2$ & 2 &14 &4& $\omega_1 + 2\omega_2$ & 14  \\
\hline
\end{tabular} }
\caption{Weyl vectors and their norms for all simple Lie algebras}
\label{Weylvec}
\end{table}

\subsection*{Casimir constant}

Let $U$ be an irreducible representation of $\fg$ with highest weight $\lambda$. Then by Schur's Lemma, the Casimir element $\Omega_{\fg}$ of $\fg$  
acts on $U$ by the scalar
\[
 {\;C_U = (\lambda, \lambda) + 2(\lambda, \rho)\;}.
\]  
 
Let $\langle\cdot,\cdot\rangle$ be the standard normalized form on $\fh^*$ with $\langle\phi,\phi\rangle = 2$ (where $\phi$ is the highest root). The relationship is
\[
(\alpha, \beta) = \frac{1}{2h^\vee} \, \langle\alpha, \beta\rangle,
\]
where $h^\vee$ is the {dual Coxeter number}. Therefore
\[
C_U = \frac{1}{2h^\vee}\bigl(\langle\lambda,\lambda\rangle + 2\langle\lambda,\rho\rangle\bigr).
\]
The  Casimir constants of the irreducible representations with $0<l_U<1$ for simple Lie algebras are listed in the last column of Table \ref{Index}.
 
\begin{example}[Casimir constant for $B_3$  on the spin representation]
Let $\mathfrak{g} = \mathfrak{so}_7$ be the Lie algebra of type $B_3$. The spin representation $U$ is the irreducible representation with highest weight
\[
 \omega_3 = \frac{\varepsilon_1 + \varepsilon_2 + \varepsilon_3}{2},
\]
where $\{\varepsilon_i\}$ is an orthonormal basis of $\mathfrak{h}^*$ with respect to the normalized form $\langle\cdot,\cdot\rangle$. 
About this representation, we have the following data   
\[
\dim \fg=21, ~~ \dim U=8, ~~ h^\vee=5, ~~ \rho=\frac{5}{2}\varepsilon_1+\frac{3}{2} \varepsilon_2+\frac{1}{2} \varepsilon_1.
\]
Then this implies 
\[
C_U= \frac{1}{2h^\vee}\bigl(\langle\omega_3, \omega_3\rangle + 2\langle\omega_3, \rho\rangle\bigr)=\frac{1}{2\cdot 5}\cdot (\frac{3}{4}+2\cdot \frac{9}{4})=\frac{21}{40}.
\]
By $\displaystyle C_U=\frac{l_U\cdot \dim \fg}{\dim U}$, a direct computation yields $\displaystyle l_U=\frac{1}{5}$.  
\end{example}

\subsection*{Automorphism of Lie algebra }

Let $\fg$ be a complex simple Lie algebra with Dynkin diagram $\dy$. A diagram automorphism on $\dy$ of order $r$ leads to the affine Dynkin diagram $\dy^r$.
By assigning integers to the vertices, Kac  uses $\dy^r$ to represent finite order $\fg$-automorphisms.
His result first appears as a brief announcement
in \cite{Kac69} and later as a detailed exposition in \cite[Ch. X-5]{Hel} and \cite[Ch. 8]{ka90}.

Let $\tilde{\fh}$ be a Cartan subalgebra of $\fg$.  A simple system
$\Pi_\fg$ of $\Delta(\fg,\tilde{\fh})$ leads to the Dynkin diagram $\dy$ whose vertices represent $\Pi_\fg$.
Consider a diagram symmetry $d$ on $\dy$ of order $r \in\{1, 2\}$ (we ignore $r = 3$ since we mainly consider order 4 automorphisms).
It lifts to an $\tilde{\fh}$-stable $\fg$-automorphism of order $r$ whose eigenspaces
form a $\bz_r$-grading $\fg =\sum_{i\in \bz_r}\fg^i$,  and $\fh = \tilde{\fh}\cap  \fg^{\bar{0}}$ is a Cartan subalgebra of $\fg^{\bar{0}}$ .
Let $\Pi_{\bar{0}}\subset \Delta(\fg^{\bar{0}}, \fh)$  be the simple system and root system of $\fg^{\bar{0}}$, obtained by restricting
$\Pi_\fg \subset \Delta(\fg,\tilde{\fh})$ to $\fh$. Construct $\dy^r$ by letting its vertices be
\[
\dy^r = \Pi_{\bar{0}}\cup \{\mbox{lowest weight of the adjoint $\fg^{\bar{0}}$-representation on $\fg^{\bar{1}}$}\}.
\]
Also, adjacent vertices of $\dy^r$ correspond to the simple rank two subalgebra indicated
by the subdiagram (except for $A^1_1$). For $r = 1$, we have $\fg = \fg^{\bar{0}} = \fg^{\bar{1}}$ and $\fh = \tilde{\fh}$,
so $\dy^1$ is just the extended Dynkin diagram $\Pi_\fg \cup  \{\mbox{lowest root}\}$. For $r = 2$, the
subdiagram $\Pi_{\bar{0}}\subset  \dy^2$ consists of the $d$-orbits of vertices of $\Pi_\fg$, and the edges of $\Pi_{\bar{0}}$
are drawn according to the root lengths
\[
\mbox{adjacent 2-element $d$-orbit  $ <$ nonadjacent 2-element $d$-orbit
$<$ vertices fixed by $d$.} \]
In each case of $\dy^r$, we have added one extra vertex to $\Pi$, so $\#(\dy^r) = \dim \fh + 1$.
Hence there is a nontrivial linear combination
$\sum_{\dy^r} a_\alpha \alpha = 0$, and the coefficients
$\{a_\alpha\}_{\dy^r}$ are the unique ones which are positive integers without nontrivial common factor.
About the affine diagram $\dy^r$ and $a_\al\in\dy^r$, see \cite[Ch.X--5]{Hel}, \cite[Ch.4]{ka90}  or \cite[p.1726]{ch}.

\begin{theorem}[Kac]
Let $\{n_\alpha\}_{\dy^r}$ be nonnegative integers without nontrivial common
factor, let $m = r\sum_{\dy^r}a_\alpha n_\alpha$,  and let $\displaystyle c_\alpha = (\exp \frac{2\pi i}{m}) n_\alpha$.
Fix $\{X_\alpha\}_{\dy^r}$ from the simple root vectors of $\fg^{\bar{0}}$ and a lowest root vector of $\fg^{\bar{1}}$. Then $\{X_\alpha\}_{\dy^r}$ generates $\fg$.

\noindent{(1)} The data $\{c_\alpha\}_{\dy^r}$ represents a unique $\fg$-automorphism $\sigma$ of order $m$ by $\sigma X_\alpha =
c_\alpha X_\alpha$ for all $\alpha\in \dy^r$.

\noindent{(2)} The fixed point set of $\si$ is a reductive algebra $\fg^\si = [\fg^\si, \fg^\si]+Z(\fg^\si)$. The Dynkin
diagram of the semisimple subalgebra $[\fg^\si, \fg^\si]$ is the subdiagram $\{\alpha \in\dy^r;  c_\alpha = 1\}$.
The dimension of the center $Z(\fg^\si)$ is $\#\{\alpha\in\dy^r; c_\alpha \neq 1\}-1$.

\noindent{(3)} Every finite order $\fg$-automorphism can be represented uniquely in this way.
\label{Kac}
\end{theorem}

 \subsection*{Lie superalgebra} 
 Let $L = \Lz + \Lo$ be a complex contragredient Lie superalgebra which is not a Lie
algebra, so it is one of $A(m, n), B(m, n), C(n), D(m, n), D(2, 1; \alpha), F(4)$ and $G(3)$
\cite[\S 2.5]{ka77}. We say that $L$ is of type 1 or 2 by
\[ \begin{array}{ll}
\mbox{Type 1:} &  \mbox{$\Lz$
is reductive and has a 1-dimensional center, }\\
&\mbox{its action on $\Lo$ has two irreducible factors,} \\
&\mbox{that is, $L$ is one of $A(m,n)$, $C(n)$;} \\
\mbox{Type 2:} & \mbox{$\Lz$ 
is semisimple, its action on $\Lo$
is irreducible, }\\
& \mbox{that is, $L$ is one of  $B(m, n), D(m, n), D(2, 1; \alpha), F(4)$ and $G(3)$.
} \end{array}
\]
The  \emph{Killing form} on $L$ is denoted by 
\[
\kapl(X, Y ) = \mathrm{str}(\ad (X), \ad (Y )) , ~~   \forall X, Y \in L. \]
This bilinear form is consistent, supersymmetric
and invariant. 

\begin{proposition}[\cite{ka77}]
\begin{itemize}
\item[(1)]   A finite dimensional contragredient Lie superalgebra $L$ with a nondegenerate Killing form
is of the type $ A(m, n)~ (m \neq n)$, $B(m, n)$, $C(n)$, $D(m, n) ~ (m \neq n+1)$,
$F(4)$ or $G(3)$.
\item[(2)] A  finite dimensional contragredient Lie superalgebra $L$ with a zero Killing form is of the
type $A(n, n), D(n + 1, n)$ or $D(2, 1; \alpha),$.
\end{itemize}
\end{proposition}
In this paper, when we consider the Dirac constant or Kostant constant for Lie superalgebra, we usually assume $L$ to be a contragredient Lie superalgebra with non-degenerate Killing form. 
This means   $L\neq A(n, n), D(n + 1, n)$ or $D(2, 1; \alpha)$.


\section{Generalized symmetric pairs  with symplectic subspace}

Let $\fg$ be a complex semisimple Lie algebra, and let $\sigma$ be an automorphism of $\fg$ of order $n$. Then $\fg$ has a decomposition
\[ \displaystyle \fg=\sum_{k=0}^{n-1}\fg(\omega^k),\]
where $\fg(\omega^k)=\{x\in\fg|\sigma(x)=\omega^k x\}$ and $\displaystyle  \omega=\exp \frac{2\pi i}{m}$ is a primitive $n$th root of unity. By Theorem \ref{Kac}, the following results are standard.

\begin{proposition}
\noindent{(1)} Let $\fh(1)$ be a Cartan subalgebra of $\fg(1)$. Then $z(\fh(1))=\{x\in\fg\mid[x,\fh(1)]=0\}$ is a Cartan subalgebra of $\fg$ (see Lemma 5.3 in  \cite[Chapter X]{Hel}).

\noindent{(2)} There is a regular element $H$ of $\fg$ in $\fh(1)$, which defines a positive root system $\Delta^+$. Since $\sigma(H)=H$, we have $\sigma(\Delta^+)=\Delta^+$, i.e., $\sigma$ is induced by a diagram automorphism $\sigma_0$ of $\fg$.

\noindent{(3)} Any diagram automorphism $\sigma_0$ defines a Lie algebra automorphism, also denoted by $\sigma_0$, and $\sigma_0$ is called a regular automorphism of $\fg$. Note that $\sigma_0$ is not unique.

\noindent{(4)} $\sigma_0^{-1}\sigma$ is an isomorphism which fixes each $H\in\fh$. Therefore, $\sigma_0^{-1}\sigma=\mathrm{e}^{\ad H}$ for some $H\in\fh$. Hence, $\sigma=\sigma_0\mathrm{e}^{\ad H}$.

\noindent{(5)} Set $H=\sum\limits_{i=0}^{n-1}H_i$. Note that $\sigma_0(H_i)=\omega^iH_i$. Then $\sigma_0\mathrm{e}^{\ad H_i}=\mathrm{e}^{\omega^i\ad H_i}\sigma_0$. For $i\neq 0$, set $k=\frac{1}{1-\omega^i}$, we have $\sigma_0\mathrm{e}^{\ad H_i}=\mathrm{e}^{-(1-k)\ad H_i}\sigma_0\mathrm{e}^{(1-k)\ad H_i}$. Replacing the regular automorphism $\sigma_0$ by $\mathrm{e}^{-(1-k)\ad H_i}\sigma_0\mathrm{e}^{(1-k)\ad H_i}$, we have $\sigma=\sigma_0\mathrm{e}^{\ad H_0}$.

In particular, if $\sigma$ is an inner automorphism of order $m$, then $\sigma=\mathrm{e}^{\ad H_0}$ for some $H_0\in\fg(1)$, and $\fg(1)$ contains a Cartan subalgebra of $\fg$.
\label{prop31}
\end{proposition}

Since $\fg$ is a complex simple Lie algebra with Killing form $\kapg$, which is a non-degenerate invariant symmetric bilinear form on $\fg$.
Let $\si\in \aut(\fg)$ have order 4, and we assume that $B$ is $\si$-invariant, that is, $\kapg(\si X, \si Y)=\kapg(X,Y)$ for all $X, Y\in \fg$.
With respect to $\si$, $\fg$ has the decomposition
\[
\fg=\fgaa+\fgbb+\fgcc+\fgdd.
\]

\begin{lemma}
 If  $\fgbb$ has a $\fgaa$-invariant symplectic form, then $\fgbb+\fgdd$ has a $\fgaa+\fgcc$-invariant symplectic form.
 \label{lem1}
\end{lemma}

\begin{proof}
Let $\langle\cdot, \cdot \rangle: \fgbb\times \fgbb\rightarrow \bc$ be the $\fgaa$-invariant symplectic form.
Define the map $f: \fgbb\rightarrow \fgdd, ~u \longmapsto f(u)$ by
\[
\langle u, v \rangle =\kapg(f(u), v), ~~ \forall v\in \fgbb.
\]

We first prove that $f$ is well-defined.
If there exist $w_1, w_2\in \fgdd$ such that $\langle u, v\rangle =\kapg(w_1, v)=\kapg(w_2, v), \forall v\in \fgbb$,
then $\kapg(w_1-w_2, v=0), \forall v\in \fgbb$.
Also we know that $\kapg(\fgbb, \fgaa)=\kapg(\fgbb,\fgbb)=\kapg(\fgbb, \fgcc)=0$, then $w_1-w_2=0$ because $\kapg$ is non-degenerate.
This implies that $f$ is well-defined.

Next we prove that $f$ is injective.
If $u_1, u_2\in \fgbb$ and $f(u_1)=f(u_2)$, then 
\[\langle u_1, v\rangle=\kapg(f(u_1), v)=\kapg(f(u_2), v)=\langle u_2, v\rangle, \forall v\in \fgbb.
\]
Since $\langle \cdot, \cdot \rangle $ is non-degenerate on $\fgbb$, then $u_1=u_2$.

Furthermore, since $\kapg: \fgbb\times \fgdd\rightarrow \bc$ is non-degenerate, then $\dim \fgbb=\dim\fgdd$.  This implies that $f$ is a bijection.

Define $\langle \cdot, \cdot \rangle$ on $\fgbb+\fgdd$ as follows:
\begin{itemize}
\item[(1)] $\langle f(u), f(v)\rangle=\langle u, v\rangle$,  $ \forall f(u), f(v)\in \fgdd$.
\item[(2)] $\langle u, f(v)\rangle =0$, $\forall u\in \fgbb, f(v)\in \fgdd$.
\end{itemize}

We first prove that this symplectic form on $\fgbb+\fgdd$ is non-degenerate.
If for any $f(v)\in \fgdd$, we have$\langle f(u),f(v)\rangle=0$. Since $f$ is a bijection, then $\langle u, v\rangle=0, \forall v\in \fgbb$.
This implies $u=0$ and hence $f(u)=0$.
So this symplectic form on $\fgbb+\fgdd$ is non-degenerate.

Next we prove that it is $\fgaa$-invariant.
Since $[\fgaa, \fgbb]\subset\fgbb$, $\langle \fgbb, \fgdd\rangle=0$ and $\langle \cdot, \cdot \rangle $ is $\fgaa$-invariant on $\fgbb$, we just need to prove that $\langle [f(u), x], f(v)\rangle =\langle f(u), [x, f(v)]\rangle$
holds for any $f(u), f(v)\in \fgdd$ and $x\in \fgaa$.

We claim that $f([u, x])=[f(u), x], \forall x\in \fgaa, u\in \fgdd$.
For any $v\in \fgbb$, we have \[
\kapg([f(u), x], v)=\kapg(f(u), [x, v])=\langle u, [x, v]\rangle=\langle [u, x], v\rangle=\kapg(f([u,x]), v).
\]
This implies that  $f([u, x])=[f(u), x]$ as claimed. Similarly, we have $f([x, u])=f([u,-x])=[f(u), -x]=[x, f(u)]$.

Therefore, for any $f(u), f(v)\in \fgdd$ and $x\in \fgaa$, we have \begin{eqnarray*}
\langle [f(u), x], f(v)\rangle&=&\langle f[u, x], f(v)\rangle =\langle [u, x], v\rangle=\langle u, [x,v]\rangle\\
&=&\langle f(u), f[x, v]\rangle=\langle f(u), [x, f(v)]\rangle.
\end{eqnarray*}

Finally, we need to prove that $\langle \cdot, \cdot \rangle$ is $\fgcc$-invariant.
Since $\langle \fgbb, \fgdd \rangle=0$, then we just need to prove that
$\langle[u, y], f(v)\rangle=\langle u, [y, f(v)]\rangle, ~\forall y\in \fgcc, u\in \fgbb, f(v)\in \fgdd$.

We claim that $f([y, w])=-[y, f^{-1}(w)], \forall y\in \fgcc, w\in \fgdd$.
By the definition of $f$, it is easy to get that $\kapg(f(u), v)+\kapg(u, f(v))=\langle u, v\rangle+\langle v, u\rangle =0, \forall u, v\in \fgbb$.
Then \begin{equation}\label{inv1}
\kapg(f([y, w]), u)=-\kapg([y,w], f(u))=-\kapg(f(u), [y, w])=-\langle u, [y,w]\rangle,
     \end{equation}
and
\begin{eqnarray}
\kapg([y, f^{-1}(w)], u) &=&-\kapg(f^{-1}(w), [y, u])=-\kapg([y,u], f^{-1}(w)) \nonumber \\
 &=&-\langle f^{-1}([y, u]), f^{-1}(w)\rangle= -\langle[y,u], w\rangle=\langle u, [y, w]\rangle.
\label{inv2}
\end{eqnarray}
By (\ref{inv1}) and (\ref{inv2}), we have $f([y, w])=-[y, f^{-1}(w)], \forall y\in \fgcc, w\in \fgdd$ as claimed.

For any $y\in \fgcc, u\in \fgbb, f(v)\in \fgdd$, we have
\begin{eqnarray*}
\langle[u, y], f(v)\rangle&=&\langle f^{-1}([u, y]), v\rangle=\kapg([u, y], v)=\kapg(u, [y,v])=-\kapg(u, f[y, f(v)])\\
&=&-\kapg( f([y, f(v)]), u)=-\langle [y, f(v)], u\rangle=\langle u, [y, f(v)]\rangle.
\end{eqnarray*}
Then $\langle \cdot, \cdot \rangle$ is $\fgcc$-invariant.
Hence we get that $\fgbb+\fgdd$ has a $\fgaa+\fgcc$-invariant symplectic form.
\end{proof}

\begin{lemma}
If  $\fgbb$ has a $\fgaa$-invariant symplectic form, then $\fgaa+\fgcc$ has a one-dimensional center.
\label{lem2}
\end{lemma}

\begin{proof}
By Lemma \ref{lem1}, we have that $\fgbb+\fgdd$ has a $\fgaa+\fgcc$-invariant symplectic form.
 Then $(\fg, \si^2)$ is a symplectic symmetric pair, which implies $\dim Z(\fgaa+\fgcc)=1$ by \cite{bi}.
\end{proof}

\bigskip 

\noindent{Proof of Theorem \ref{main1}}:
When $\fgbb$ has a $\fgaa$-invariant symplectic form, then $\fgaa+\fgcc$ has a one-dimensional center by Lemma \ref{lem2}.
Let $z\in Z(\fgaa+\fgcc)$ be nonzero.
We claim that $z\in \fgaa$ or $\fgcc$.
Write $z=z'+z''$ with $z'\in \fgaa$ and $z''\in \fgcc$. Then $\forall x\in \fgaa $, we have
\[
0=[z, x]=[z'+z'', x]=[z', x]+[z'', x],
\]
which implies that $[z', x]=[z'', x]=0$ because $[z', x]\in \fgaa$ and $[z'', x]\in \fgcc$.
Therefore $z\in Z(\fgaa)$.
Similarly, we have $[z', y]=[z'', y]=0 $ for any $y\in \fgcc$.
Then $z', z''\in Z(\fgaa+\fgcc)$.
Since $\dim Z(\fgaa+\fgcc)=1$, then $z\in \fgaa$ or $z\in\fgcc$.

Assume that $\sigma z=\ep z$ with $\ep=\pm 1$ according to $z\in \fgaa$ or $\fgcc$.
Up to scalar, the symplectic structure on $\fgbb+\fgdd$ is defined by $\langle u, v \rangle=\kapg(z, [u, v])$ for $u, v\in \fgbb+\fgdd$.
Given $u\in \fgbb, v\in \fgdd$, we have
\[
\langle u, v\rangle =\kapg(z, [u,v])= \kapg(\sigma z, [\sigma u, \sigma v])=\ep \kapg(z, [u, v])=\ep \langle u, v\rangle.
\]
If $\ep=-1$, we have $\langle \fgbb, \fgdd\rangle =0$.
But $\langle \cdot, \cdot \rangle$ is non-degenerate on $\fgbb+\fgdd$, then $\fgbb$ has a symplectic form.

If $\ep=1$, we can prove that $\langle \fgbb, \fgbb\rangle =\langle \fgdd, \fgdd\rangle =0$ similarly.
Then this prove Theorem \ref{main1}.
 $\hfill$$\Box$

\bigskip

\noindent{Proof of Theorem \ref{main2}}:
If $\sigma$ is inner, then by Proposition \ref{prop31}, the center of $\fgaa+\fgcc$ should be in the Cartan subalgebra $\fh$, which is in $\fgaa$.
Thus $\fgaa+\fgcc$ has no center  in $\fgcc$.
Therefore we just need to consider the cases when $\sigma$ is outer.
For these cases, we compute $\fg^\si=\fgaa$ and $\fg^{\si^2}=\fgaa+\fgcc$.
Also note that $(\fgaa+\fgcc, \fgaa)$ is a symmetric pair.
We list the results in the following Table \ref{table2}.

\begin{table}[ht!] \renewcommand{\arraystretch}{2.5}\small
\begin{center}
\begin{tabular}{|c|c|c|c|c| }
  \hline
  $\fg$ & $\si$ & $\fgaa$ & $\fgaa+\fgcc$ & center   \\
  \hline
   &\includegraphics[width=3.5cm]{A4.png}  &  {$ C_{n-1}\oplus \bc
  $} & {$A_{2n-3}\oplus A_1\oplus \bc$ } &  { $z\in \fgcc$} \\
    \cline{2-5}
{ $A_{2n-1}
$}  &\includegraphics[width=3.5cm]{A5.png}  &  { $ D_p\oplus C_{n-p} $} & {$A_{2p-1}\oplus A_{2n-2p-1}\oplus \bc$ } &   { $z\in \fgcc$ }\\
 \cline{2-5}
  &\includegraphics[width=3.5cm]{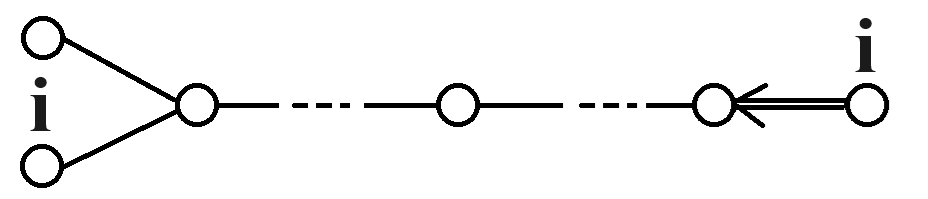}  &  { $A_{n-1}\oplus \bc$} & {$A_{n-1}\oplus A_{n-1}\oplus \bc$ } &  $z\in \fgaa$ \\
 \hline
{ $ A_{2n} $}  &\includegraphics[width=3.5cm]{A7.png}  &   { $ C_p\oplus B_{n-p}  $} & {$A_{2p-1}\oplus A_{2n-2p}\oplus \bc$ } &   {$z\in \fgcc$ }\\
    \cline{2-5}
 &\includegraphics[width=3.5cm]{A8.png}  &  { $ C_n
 $} & {$A_{2n-1} \oplus \bc$ } & { $z\in \fgcc$} \\
  \hline
 &\multirow{2}{*}{\includegraphics[width=3.5cm]{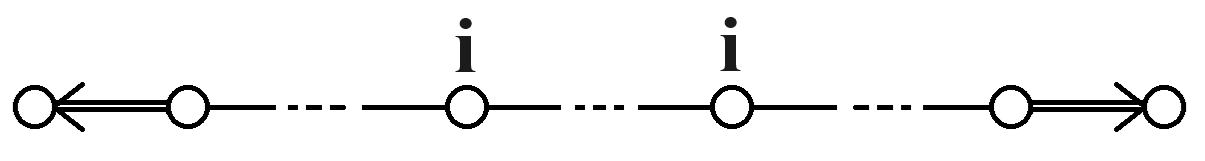}}  &     { $  B_{p}\oplus  A_q\oplus B_{n-p-q-2}\oplus  \bc$} &   $ D_{q+1}\oplus  D_{n-q-1 } $  &   $\dim Z=0 $ \\
\cline{3-5}
 & &{ $B_{p}+B_{n-p-2}+ \bc$} &  { $ D_{n-1}\oplus \bc $} &   { $ z\in \fgaa$} \\
   \cline{2-5}
{ $ D_{n} $}     &\includegraphics[width=3.5cm]{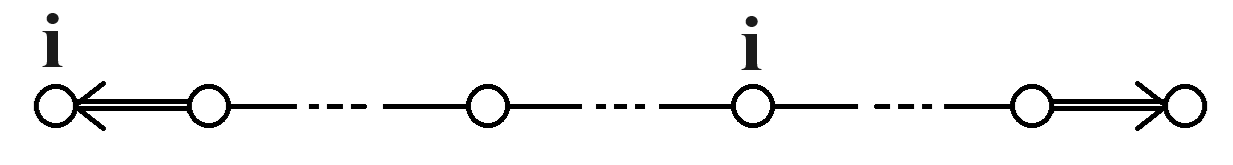}  &  { $  A_{p}\oplus B_{n-p-2}\oplus  \bc  $} &   $ D_{p+1}\oplus D_{n-p-1}$  &   $\dim Z=0 $ \\
      \cline{2-5}
 &\includegraphics[width=3.5cm]{D2-3.png}  &  { $  A_{n-2}\oplus  \bc  $} &   $  D_{n-1}\oplus \bc$  &     {$z\in \fgcc $} \\
  \hline
&\includegraphics[width=2.6cm]{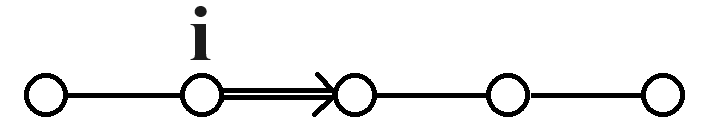}  &  { $  A_1\oplus A_3  $} &   $ A_1\oplus A_5 $  &    $\dim Z=0 $\\
   \cline{2-5}
{ $ E_6 $}     &\includegraphics[width=2.6cm]{E6-22.png}  &  { $B_{3}\oplus A_{1} $} &   $ D_{5}\oplus \bc$  &    {$z\in \fgcc $} \\
      \cline{2-5}
 &\includegraphics[width=2.6cm]{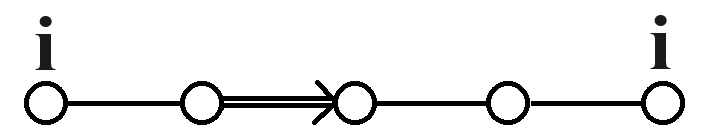}  &  { $ C_3\oplus  \bc  $} &   $  A_5\oplus A_1 $  &   $\dim Z=0 $  \\
  \hline
      \end{tabular}
\caption{  $\fgaa$ and $\fgaa+\fgcc$  }
\label{table2}
\end{center}
 \end{table}

 \setcounter{equation}{0}

Now we explain Table \ref{table2}.
In the second column of Table \ref{table2}, the diagrams satisfy Theorem \ref{Kac} and $\fgaa$ in the 3rd column is got by Theorem \ref{Kac}(b).
Next we illustrate the case of $D_n$ as an example to explain how to get $\fgaa+\fgcc$.
The example is showed in Table \ref{table3}.

\begin{center}
\begin{table}[ht]   \renewcommand{\arraystretch}{2.5}
\begin{tabular}{ |c|c|c|c| }
  \hline
  $\si$ & Vogan diagram of $\si^2$ & equivalent diagram & $\fgaa+\fgcc$  \\
  \hline
 \multirow{2}{*}{\includegraphics[width=3.5cm]{D2-1.png}}  &     \multirow{2}{*}{\includegraphics[width=3.5cm]{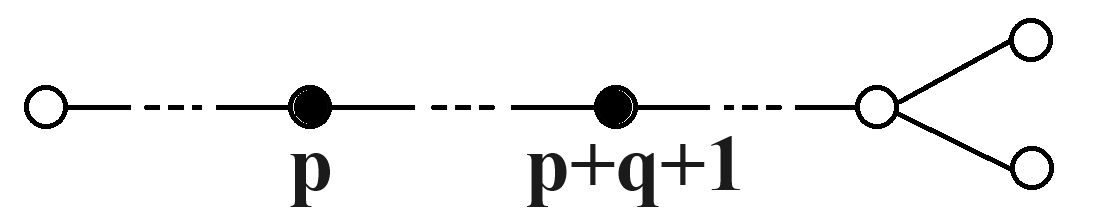}} & \includegraphics[width=3.5cm]{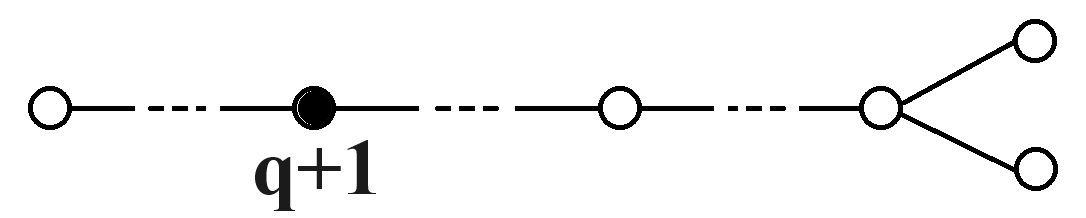} & $ D_{q+1}+ D_{n-q-1} $ \\
\cline{3-4}
 &&  \includegraphics[width=3.5cm]{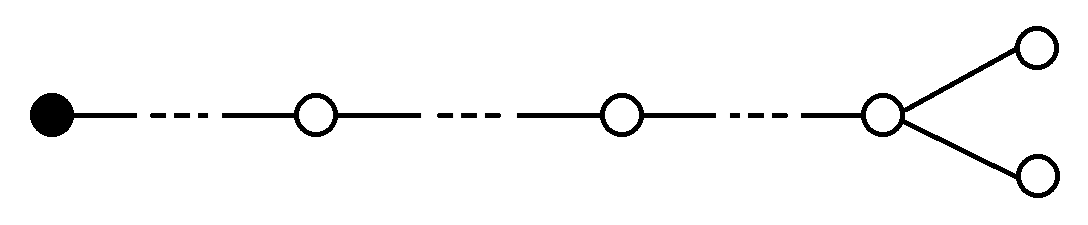}& { $ D_{n-1}+\bc $}\\
   \hline
 \includegraphics[width=3.5cm]{D2-2.png}  & \includegraphics[width=3.5cm]{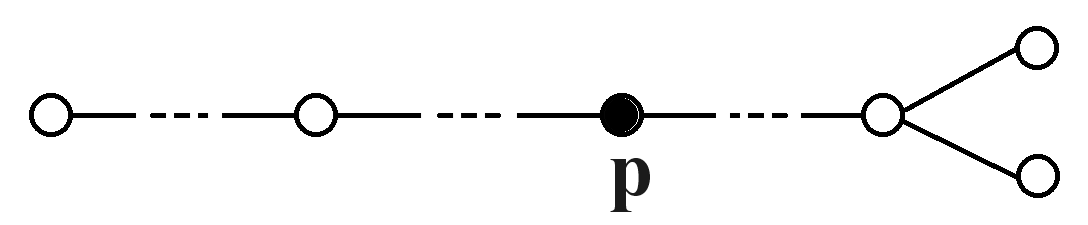} &    \includegraphics[width=3.5cm]{D104.png}& { $ D_{p+1}+D_{n-p-1}$ }\\
  \hline
  \includegraphics[width=3.5cm]{D2-3.png}  & \includegraphics[width=3.5cm]{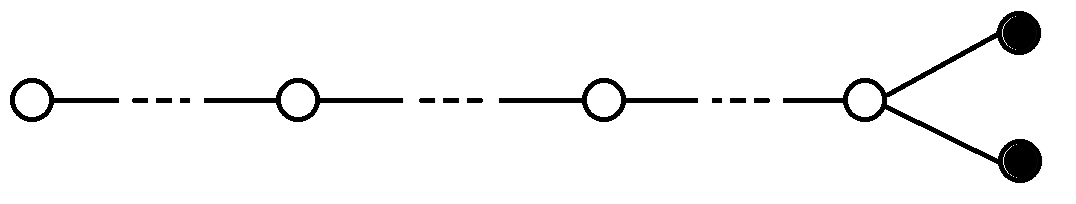} &    \includegraphics[width=3.5cm]{D103.png}& { $ D_{n-1}+\bc $}\\
  \hline
 \end{tabular}
\caption{  $\fgaa+\fgcc$ for $D_n$ }
\label{table3}
 \end{table}
\end{center}

By Theorem \ref{Kac}, it is easy to get the Vogan diagram of $\si^2$ as listed in the 2nd column of Table \ref{table3}.  In 3rd column of Table \ref{table3}, the equivalent diagram of $\si$ is got by \cite[Table 1]{hcc}.
Then we get $\fgaa+\fgcc$ in the 4th column of Table \ref{table3}.

Table \ref{table1} lists exactly the pairs from Table \ref{table2} that satisfy the condition $z \in \fgcc$ as indicated in the last column of Table \ref{table2}.
Hence we prove Theorem \ref{main2}.
$\hfill$$\Box$

\section{Generalized symmetric pairs and Lie superalgebras}

In this section, we prove Theorem \ref{main}. The proof relies on a complicated case-by-case computation concerning root systems. We illustrate Table \ref{table1}(a) as an example to demonstrate the proof explicitly. We also briefly discuss the two exceptional cases in Table \ref{table1}(f) and Table \ref{table1}(b) with $\fg = A_5$.

First we use markings on the Dynkin diagram to give another description of order 4 outer automorphisms of $\fs\fl(n,\bc)$.
This will make
the computation easy in the following part of this section.
Then we prove Theorem \ref{main} for $\fg=\fs\fl(n, \bc)$.

\subsection{Finite order automorphisms on $\fs\fl(n,\bc)$}
\setcounter{equation}{0}

For complex simple Lie algebra $\fg$,  let $\dy$ be its Dynkin diagram.
Let $\cU$ be the multiplicative group of complex
numbers of finite order, namely
\[
\cU=\{z\in \bc; z^m=1 \mbox{ for some } m\in \bn\}.
\]

A \emph{marking} on a Dynkin diagram $\dy$ is a pair $(c, d)$,  where \[
c: \dy \longrightarrow \cU, \qquad d \in \aut(\dy).
\]
 It is said to \emph{represent} a finite order $\fg$-automorphism $\si$ if there exist vector space
automorphisms $\si_c$, $\si_d$ on $\sum_{\alpha\in\dy} \fg_\alpha$ and basis $\{X_\alpha\in \fg_\alpha\}_{\alpha\in\dy}$
 such that
 \begin{equation}
\left\{ \begin{array}{ll}
\mbox{(a)} & \mbox{ $\si_c$ acts as multiplication by $c_\al$ on $\fg_\al$};\\
\mbox{(b)} & \mbox{  $\si_d X_\al = X_{d\al}$ for all $\al\in  \dy$}; \\
\mbox{(c)} & \mbox{  $\si = \si_c \cdot \si_d$ on  $\sum_{\alpha\in\dy} \fg_\alpha$}.
\end{array}\right.
\end{equation}
The Vogan diagrams \cite{kn} are the markings on $\dy$ which satisfy $c_\al = \pm 1$, $\mathrm{ord}(d) =
1, 2$, and also $c_\al = 1$ whenever $d\al \neq\al$.

Let $\aut(\fg)$ denote automorphisms on $\fg$, and let $\inn(\fg)$ denote inner automorphisms on $\fg$.
It is well-known (i.e. \cite[Thm.7.8]{kn}) that
\[
\aut(\fg)/ \inn(\fg)\cong \aut(\dy).
\]
Recall that $\si$ is an outer automorphism of $\fg$ if $\si\notin\inn(\fg)$ and
naturally we have the following result.
\begin{lemma}
If $\si$ is an finite order automorphism of $\fg$, and $\si$ is represented by a marking $(c,d)$ on $\dy$, then
$\si$ is outer if and only if $d\neq 1$.
\label{lem:outer}
\end{lemma}

Now consider the case $\fg=\fs\fl(n,\bc)$. We label the roots of the Dynkin diagram $\dy$ of $\fg=\fs\fl(n,\bc)$ by
\[
\ep_1-\ep_2, \ep_2-\ep_3, \ldots, \ep_{n-1}-\ep_{n}.
\]
Let $\displaystyle \lfloor \frac{n}{2} \rfloor$ stand for the maximal integer which is not larger than $\displaystyle \frac{n}{2}$.
Now for $\displaystyle  1\leq k \leq \lfloor \frac{n}{2} \rfloor$, we define the marking $(c_k, d)$ on  $\dy$
by:
\begin{equation}
\left\{ \begin{array}{l}
c_{\ep_{k}-\ep_{k+1}}=i; \quad c_{\ep_{n-k}-\ep_{n+1-k}}=i; \\
c_{\ep_s-\ep_{s+1}}=1, \mbox{ for $s\neq k, n -k$}; \\
d: \ep_t-\ep_{t+1} \longmapsto \ep_{n -t}-\ep_{n+1-t}, \mbox{ for $1\leq t\leq n-1$}.
\end{array}\right.
\label{mark}
\end{equation}
Next we will relate the markings  $\displaystyle  \{(c_k, d); 1\leq k \leq \lfloor \frac{n}{2}\rfloor \}$ with Kac's classification of finite order automorphisms  of $\fs\fl(n,\bc)$.
By Theorem \ref{Kac}, we know that
up to conjugate class, all order 4 automorphisms are represented in Figure \ref{autA}.
Moreover, the automorphisms represented by Figure \ref{autA}(f-h) are inner.

\begin{figure}[ht]
\begin{center}
\includegraphics[width=10cm]{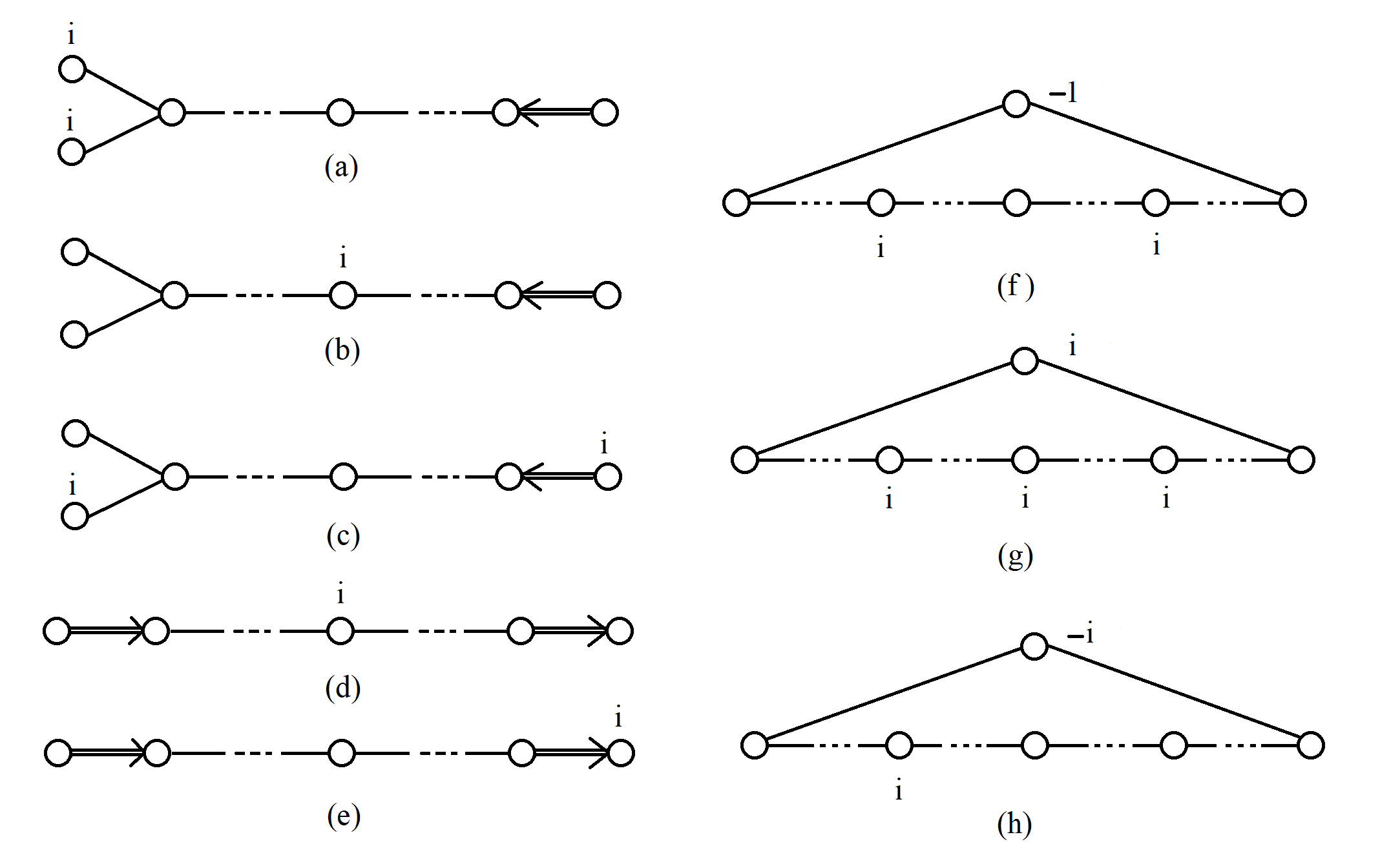}
\caption{Order 4 outer automorphisms of $A_{n-1}$}
\label{autA}
\end{center}
\end{figure}

\begin{theorem}
Up to conjugate class, the markings $\displaystyle \{(c_k, d); 1\leq k \leq \lfloor \frac{n}{2}\rfloor \}$ defined in (\ref{mark}) represent all the order 4 outer automorphisms of $\fs\fl(n,\bc)$.
Also we have the following equivalence of  automorphisms in Figure \ref{equiauto}.

\begin{figure}[ht]
\begin{center}
\includegraphics[width=10cm]{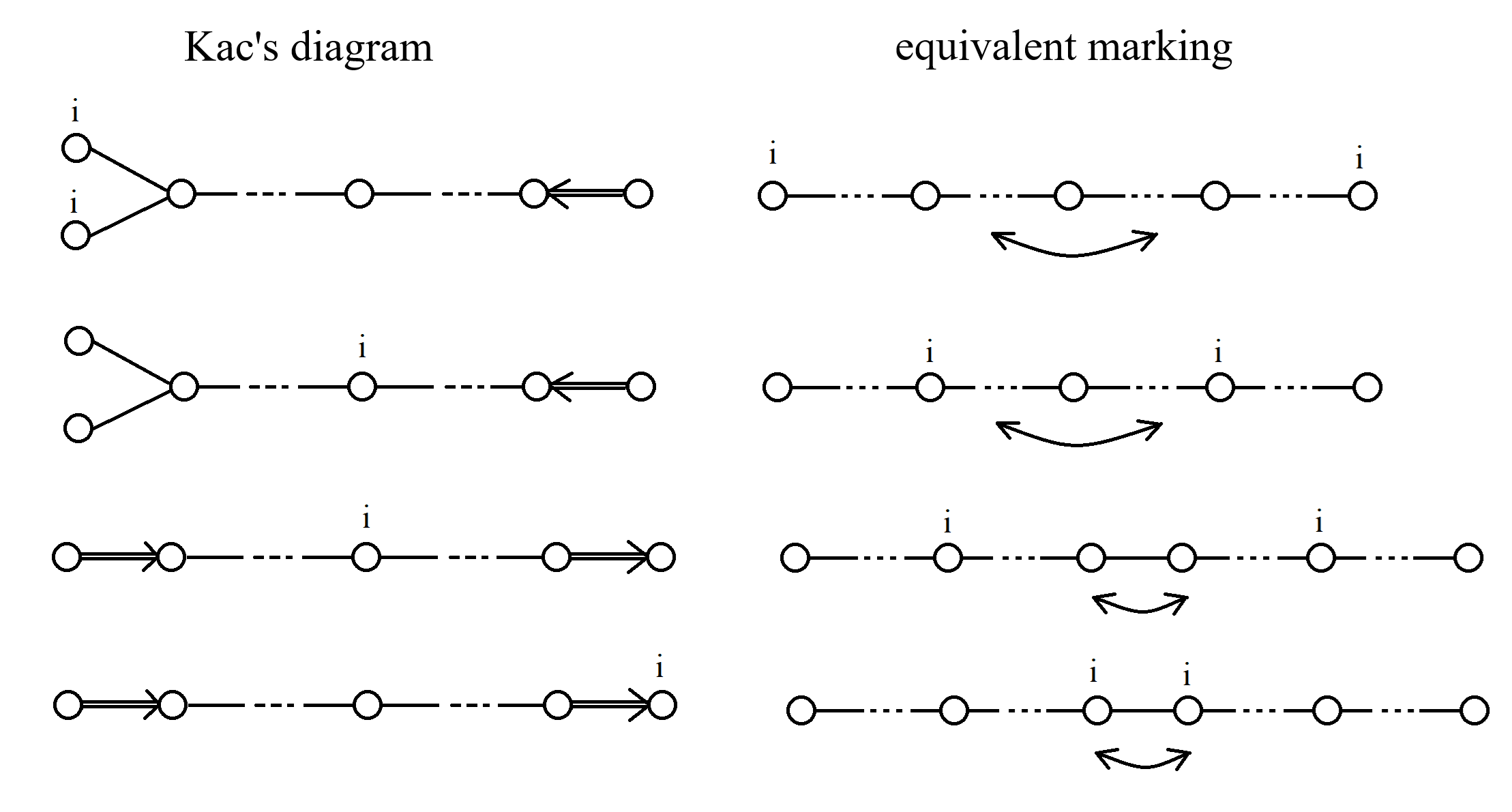}
\caption{Equivalent  description of automorphisms of $A_{n-1}$}
\label{equiauto}
\end{center}
\end{figure}
\label{ordfour}
\end{theorem}

\begin{proof}
By \cite[Thm.4.2]{tams}, we know that each marking $(c_k,d)$ represents an automorphism of $\fs\fl(n,\bc)$.
It is easy to verify that the order of all these markings is 4. By Lemma \ref{lem:outer}, we know that the automorphism
represented by the marking $(c_k,d)$ is outer because $d\neq 1$.
Now we prove that they represent all the outer automorphism up to conjugate class.

Assume $\si_k$ is the automorphism of $\fg$ which is represented by $(c_k,d)$.
Let $E_{s,t}$ be the matrix with 1 at $(s,t)$-entry and 0 elsewhere.
Then $E_{s,t}$ $(s\neq t)$ is a root vector with respect to the root $\ep_s-\ep_t$.
Since $E_{s,t}=[[\ldots[E_{s,s+1}, E_{s+1,s+2}], \ldots], E_{t-1,t}]$ and $\si_k E_{s,t}=[[\ldots[\si_k E_{s,s+1}, \si_k E_{s+1,s+2}], \ldots], \si_k E_{t-1,t}]$, then we have
\begin{equation}
\si_k E_{s,t}=\left\{ \begin{array}{ll}
(-1)^{t-s-1} E_{n+1-t,n+1-s}, &  1\leq s,t\leq k, \mbox{ or } k+1\leq s,t \leq n-k, \\
                             &   \mbox{ or } n-k+1\leq s,t\leq n; \\
(-1)^{t-s} E_{n+1-t,n+1-s}, &  1\leq s \leq k,  n-k+1\leq  t \leq n, \\
                           &  \mbox{ or } n-k+1\leq  s \leq n, 1\leq t \leq k; \\
(-1)^{t-s-1} i E_{n+1-t,n+1-s}, & \mbox{ else},
\end{array}\right.
\label{sik}
\end{equation}
for $s\neq t$ and
\begin{equation}
\si_k (E_{p,p}-E_{q,q})=E_{n+1-q,n+1-q}-E_{n+1-p,n+1-p}.
\label{sik2}
\end{equation}
Let $\fg^{\si_k}=\{X\in \fg| \si_k X=X\}$.
If $\si_k X=Y$ and $\si_k Y=X$, then we have $X+Y\in \fg^{\si_k}$.
Then we compute the dimension of $\fg^{\si_k}$ as follows:
\begin{equation}
\dim \fg^{\si_k}=\left\{\begin{array}{ll}
2k^2-k+\frac{(n-2k)^2+(n-2k)}{2}, & \mbox{ when $n$ is even and $k\neq \frac{n}{2}$}; \\
(\frac{n}{2})^2, & \mbox{ when $n$ is even and $k=\frac{n}{2}$};  \\
2k^2+k+\frac{(n-2k)^2-(n-2k)}{2}, & \mbox{ when $n$ is odd}.
\end{array}\right.
\end{equation}

By Theorem \ref{Kac}, we list the information about the automorphisms represented on the affine Dynkin diagrams:

\begin{tabular}{cccc}
\hline
parity of $n$ & $\si$  & $\fgaa$ & $\dim \fgaa$   \\
\hline
even &  Figure \ref{autA}(a) & $C_{\frac{n}{2}-1}+\bc$  &   $2(\frac{n}{2}-1)^2+(\frac{n}{2}-1)+1$ \\
even &  Figure \ref{autA}(b) & $D_k+C_{\frac{n}{2}-k}$  &   $2k^2-k+2(\frac{n}{2}-k)^2+(\frac{n}{2}-k)$ \\
even &  Figure \ref{autA}(c) & $A_{\frac{n}{2}-1}+\bc$  &   $(\frac{n}{2})^2$ \\
odd &  Figure \ref{autA}(d) & $C_k+B_{\frac{n-1}{2}-k}$  &   $2k^2+k+2(\frac{n-1}{2}-k)^2+(\frac{n-1}{2}-k)$ \\
odd &  Figure \ref{autA}(e) & $C_{\frac{n-1}{2}}$  &   $2(\frac{n-1}{2})^2+(\frac{n-1}{2})$  \\
\hline
\end{tabular}

Then comparing the dimensions of the invariant subspaces of the automorphisms, we get that
\begin{itemize}
\item[(a)]
when $n$ is even,
\[
\left\{\begin{array}{ll}
(c_k, d)\cong \mbox{the automorphism in Figure \ref{autA}(a)}, &  \mbox{ for }    k=1; \\
(c_k, d)\cong \mbox{the automorphism in Figure \ref{autA}(b)}, &  \mbox{ for }    2\leq k< \frac{n }{2}; \\
(c_k, d)\cong \mbox{the automorphism in Figure \ref{autA}(c)}, &   \mbox{ for }   k=\frac{n}{2}.
\end{array}
\right.
\]
\item[(b)]
when $n$ is odd, \[
\left\{\begin{array}{ll}
(c_k, d)\cong \mbox{the automorphism in Figure \ref{autA}(d)}, &  \mbox{ for } 1\leq k< \frac{n-1}{2}; \\
(c_k, d)\cong \mbox{the automorphism in Figure \ref{autA}(e)}, &  \mbox{ for }   k= \frac{n-1}{2}.
\end{array}
\right.
\]
\end{itemize}
Then we have proved the markings $\{(c_k, d); 1\leq k \leq \lfloor \frac{n}{2}\rfloor \}$ defined in (\ref{mark}) represent all the order 4 outer automorphisms of $\fs\fl(n,\bc)$. By the way, the automorphisms represented in Figure \ref{autA}(a-e) also cover all the order 4 outer automorphisms of $\fs\fl(n,\bc)$.
Also Figure \ref{equiauto} is natural and the diagrams in the first column of Figure \ref{equiauto} are just ones in the third column of Table \ref{table1} for type A.
\end{proof}

\subsection{Proof of Theorem \ref{main} for case in Table \ref{table1}(a)}

In this part, we prove Theorem \ref{main} for $\fg=\fs\fl(2n, \bc)$.   The roots of the Dynkin diagram of $A_{2n-1}$ are labeled by
\[
\ep_1-\ep_2, \ep_2-\ep_3, \ldots, \ep_{2n-1}-\ep_{2n}.
\]
Let $\si_k$ $(1<k<n)$ be the order 4 outer automorphism of  $A_{2n-1}$ represented by the marking $(c_k, d)$ defined in (\ref{mark}).
Then by (\ref{sik}) and (\ref{sik2}), we know that the action of $\si_k$ on the root vectors is as follows:
\begin{displaymath}
\left\{ \begin{array}{ll}
\si_k E_{s,t} =(-1)^{t-s-1}  E_{2n+1-t,2n+1-s} & \mbox{ for $1\leq s, t \leq k$ or $k+1\leq s, t\leq 2n-k$ } \\
                                             & \quad \mbox{ or $2n-k+1\leq s, t\leq 2n$}; \\
\si_k E_{s,t}  =(-1)^{t-s-1} i  E_{2n+1-t,2n+1-s} & \mbox{ for $1\leq s \leq k,  k+1\leq t \leq 2n-k$}  \\
                                             & \quad \mbox{ or $ k+1\leq s \leq 2n-k, 2n-k+1\leq t \leq 2n$}; \\
\si_k E_{s,t}  =(-1)^{t-s} i  E_{2n+1-t,2n+1-s} & \mbox{ for $  k+1\leq s \leq 2n-k, 1\leq t \leq k$}  \\
                                             & \quad \mbox{ or $ 2n-k+1\leq s \leq 2n, k+1\leq t \leq 2n-k$}; \\
\si_k E_{s,t}  =(-1)^{t-s} E_{2n+1-t,2n+1-s} & \mbox{ for $1\leq s \leq k,  2n-k+1\leq t \leq 2n$}  \\
                                             & \quad \mbox{ or $ 2n-k+1\leq s \leq 2n, 1\leq t \leq k$}.
\end{array}
\right.
\end{displaymath}

Now we list a basis of $\fgaa$ and one of $\fgbb$ as follows:
\begin{itemize}
\item[(a)]  a basis of $\fgaa$:
\begin{eqnarray*}
&& \{ E_{s,t}+(-1)^{t-s-1} E_{2n+1-t,2n+1-s}\}_{1\leq s, t \leq k}   \\
&\cup & \{ E_{s,t}+(-1)^{t-s-1} E_{2n+1-t,2n+1-s}\}_{k+1\leq s, t\leq 2n-k, s+t\leq 2n+1}   \\
&\cup & \{ E_{s,t}+(-1)^{t-s} E_{2n+1-t,2n+1-s} \}_{ 1\leq s \leq k,  2n-k+1\leq t \leq 2n,  s+t<2n+1 }  \\
  &\cup & \{ E_{s,t}+(-1)^{t-s} E_{2n+1-t,2n+1-s} \}_{ 2n-k+1\leq s \leq 2n, 1\leq t \leq k,  s+t< 2n+1 } \\
  &\cup & \{ E_{s,s}-E_{s+1,s+1}+E_{2n-s, 2n-s}-E_{2n+1-s, 2n+1-s}\}_{1\leq s\leq n};
\end{eqnarray*}
\item[(b)] a basis of $\fgbb$:
\[ \{ E_{s,t} +(-1)^{t-s-1} E_{2n+1-t,2n+1-s}, E_{t,s}+(-1)^{t-s} E_{2n+1-s,2n+1-t} \}_{1\leq s \leq k,  k+1\leq t \leq 2n-k}
\]
\item[(c)] the center of $\fgaa+\fgcc$:
\[
\bc \cdot \mathrm{diag}( \underbrace{1, \ldots, 1}_{k}, \underbrace{ \lambda, \ldots, \lambda}_{2n-2k}, \underbrace{1, \ldots ,1}_{k}),    ~~\mbox{ with } \lambda=\frac{k}{n-k}.
\]
\end{itemize}
By multiplying a nonzero coefficient,  we can choose $z$ in the center of $\fgaa+\fgcc$ such that
\begin{equation}\label{zzz}
[z, E_{s,t}]= \left\{ \begin{array}{ll}
E_{s,t}, & \mbox{ if } {1\leq s \leq k,  k+1\leq t \leq 2n-k};   \\
-E_{s,t}, & \mbox{ if } {k+1\leq s \leq 2n-k, 1\leq t \leq k}.
\end{array} \right.
\end{equation}

Also we have
 \[
\dim \fgbb=4k(n-k).
\]
\begin{lemma}
The elements $ \{ E_{s,t}+(-1)^{t-s-1} E_{2n+1-t,2n+1-s}\}_{k+1\leq s, t\leq 2n-k, s+t\leq 2n+1}$ construct a Lie algebra of type $C_{n-k}$.
The other elements in the basis of $\fgaa$ construct a Lie algebra of type  $D_k$.
Also $\fgbb$ is an irreducible $\fgaa$-representation, which is isomorphic to $\bc^{2k}\otimes \bc^{2n-2k}$.
Here $\bc^{2k}$ (resp. $\bc^{2n-2k}$) stands for the standard representation of $D_k$ (resp. $C_{n-k}$).
\label{irr}
\end{lemma}

\begin{proof}
We choose a basis of a Cartan subalgebra $\fh$ of $\fgaa$ as follows:
\[
\{H_s=E_{s,s}-E_{2n+1-s,2n+1-s} \}_{1\leq s\leq k} \cup \{H'_t=E_{t,t}-E_{2n+1-t,2n+1-t} \}_{k+1\leq  t\leq n}
\]
Define $\{\xi_p\}_{1\leq p\leq k}, \{\eta_q\}_{k+1\leq q\leq n}\subset \fh^* $ to be:
\[
\left\{ \begin{array}{ll}
\xi_p(H_s)=\delta_{p,s},  & \xi_p(H'_t)=0; \\
\eta_q(H_s)=0, & \eta_q(H'_t)=\delta_{q,t}.
\end{array}\right.
\]

Then we get the root of the vectors in $\fgaa$ listed as follows:

\begin{center}
\begin{tabular}{ccc}
\hline
case  &  vector in $\fgaa$ & associated root \\
\hline
 $1\leq s, t \leq k, s\neq t $  &   $ E_{s,t}+(-1)^{t-s-1} E_{2n+1-t,2n+1-s}$ & $ \xi_s-\xi_t $\\
$k+1\leq s, t\leq 2n-k, s+t< 2n+1, s\neq t$ &  $E_{s,t}+(-1)^{t-s-1} E_{2n+1-t,2n+1-s} $  & $ \eta_s-\eta_t$ \\
$k+1\leq s, t\leq 2n-k, s+t=2n+1, s\neq t$ &  $ E_{s,t}+(-1)^{t-s-1} E_{2n+1-t,2n+1-s} $ & $2 \eta_s $ \\
$ 1\leq s \leq k,  2n-k+1\leq t \leq 2n,  s+t<2n+1 $   &  $E_{s,t}+(-1)^{t-s} E_{2n+1-t,2n+1-s}$   & $ \xi_s+\xi_{2n+1-t}$ \\
$ 2n-k+1\leq s \leq 2n, 1\leq t \leq k,  s+t< 2n+1 $ &  $E_{s,t}+(-1)^{t-s} E_{2n+1-t,2n+1-s}$ &   $-\xi_{2n+1-s}-\xi_t$ \\
\hline
 \end{tabular}
\end{center}

 We illustrate the case $1\leq s, t \leq k, s\neq t $ to prove this table.
Since \[
[H_p, E_{s,t}+(-1)^{t-s-1} E_{2n+1-t,2n+1-s}]=(\delta_{s,p}-\delta_{p,t})E_{s,t}+(-1)^{t-s-1} E_{2n+1-t,2n+1-s},
\]
and
\[
[H'_q, E_{s,t}+(-1)^{t-s-1} E_{2n+1-t,2n+1-s}]=0,
\]
then the associated root of $ E_{s,t}+(-1)^{t-s-1} E_{2n+1-t,2n+1-s}$ is $ \xi_s-\xi_t$.
By the similar computation, we get the above table.
Then we can see that \[
 \{ E_{s,t}+(-1)^{t-s-1} E_{2n+1-t,2n+1-s}\}_{k+1\leq s, t\leq 2n-k, s+t\leq 2n+1}
 \]
construct a Lie algebra of type $C_{n-k}$, and the others
construct a Lie algebra of type $D_k$. Also $\fgaa\cong D_k+C_{n-k}$.

Let $X\in \fgaa, Y\in\fgbb$. Then we have
\[
\si_k[X,Y]=[\si_k X, \si_k Y]=[X,iY]=i[X, Y],
\]
which implies $[\fgaa,\fgbb]\subset \fgbb$. That is, $\fgbb$ is a $\fgaa$-representation.

Next we prove this representation is isomorphic to $\bc^{2k}\otimes \bc^{2n-2k}$.
For $1\leq s \leq k,  k+1\leq t \leq 2n-k$, we have
\[
[H_p,  E_{s,t} +(-1)^{t-s-1} E_{2n+1-t,2n+1-s} ]= \delta_{p,s} (E_{s,t} +(-1)^{t-s-1} E_{2n+1-t,2n+1-s}),
\]
and
\[
[H'_q,  E_{s,t} +(-1)^{t-s-1} E_{2n+1-t,2n+1-s} ]= -\delta_{q,t}(E_{s,t} +(-1)^{t-s-1} E_{2n+1-t,2n+1-s}).
\]
Similarly, for $k+1\leq s \leq 2n-k, 1\leq t \leq k$,
we have
\[
[H_p,  E_{s,t} +(-1)^{t-s} E_{2n+1-t,2n+1-s} ]= -\delta_{p,t} (E_{s,t} +(-1)^{t-s} E_{2n+1-t,2n+1-s}),
\]
and
\[
[H'_q,  E_{s,t} +(-1)^{t-s} E_{2n+1-t,2n+1-s} ]= \delta_{q,t}(E_{s,t} +(-1)^{t-s} E_{2n+1-t,2n+1-s}).
\]
Then we get the weights of the vectors in $\fgbb$ listed as follows:

\begin{center}
\begin{tabular}{ccc}
\hline
case  &  vector in $\fgbb$ & associated weight \\
\hline
 $1\leq s \leq k,  k+1\leq t \leq 2n-k$  &   $ E_{s,t} +(-1)^{t-s-1} E_{2n+1-t,2n+1-s}$ & $ \xi_s-\eta_t $\\
$k+1\leq s \leq 2n-k, 1\leq t \leq k$ &  $ E_{s,t} +(-1)^{t-s} E_{2n+1-t,2n+1-s} $  & $ -\xi_t+\eta_s$  \\
\hline
 \end{tabular}
 \end{center}

By the highest weight theory of simple Lie algebra, we get that  the $\fgaa$-representation on $\fgbb$ is irreducible and isomorphic to $\bc^{2k}\otimes \bc^{2n-2k}$.
\end{proof}

Let $L=\fgaa+\fgbb$.
Consider the bilinear operator on $L\times L \longrightarrow L$ in Theorem \ref{main}.
Without loss of generality, we assume that $J=-\id$ on the simple component of $C_{n-k}$.
By direct verification, we have  \begin{equation}
[X,Y]_s=\left\{ \begin{array}{ll}
[X, Y]=XY-YX, & \mbox{$X\in \fgaa$ or $Y\in\fgaa$}; \\
J([[z, X], Y])=XY+YX,  & \mbox{ for } X, Y\in \fgbb.
\end{array}\right.
\label{scom}
\end{equation}
Here $XY$ stands for the usual matrix multiplication in $\fg$.

We illustrate an example to prove (\ref{scom}).
Suppose $1\leq s, s' \leq k,  k+1\leq t, t' \leq 2n-k$.
Then
\begin{eqnarray*}
&& [ [z, E_{s,t} +(-1)^{t-s-1} E_{2n+1-t,2n+1-s}], E_{t',s'}+(-1)^{t'-s'} E_{2n+1-s',2n+1-t'}  ]\\
&=& [   E_{s,t} +(-1)^{t-s } E_{2n+1-t,2n+1-s}, E_{t',s'}+(-1)^{t'-s'} E_{2n+1-s',2n+1-t'}  ]  \\
&=& \delta_{t, t'} (E_{s, s'}-(-1)^{t+t'-s-s'}  E_{2n+1-s',2n+1-s} )+\delta_{s,s'}(- E_{t',t}+(-1)^{t+t'-s-s'}  E_{2n+1-t,2n+1-t'})  \\
&&+\delta_{s'+t, 2n+1}((-1)^{t'-s'}E_{s, 2n+1-t'}-(-1)^{t-s} E_{t', 2n+1-s}),
\end{eqnarray*}
and
\begin{eqnarray*}
&&  (E_{s,t} +(-1)^{t-s-1} E_{2n+1-t,2n+1-s})(E_{t',s'}+(-1)^{t'-s'} E_{2n+1-s',2n+1-t'}) \\
&&+(E_{t',s'}+(-1)^{t'-s'} E_{2n+1-s',2n+1-t'}) (E_{s,t} +(-1)^{t-s-1} E_{2n+1-t,2n+1-s})\\
&=& \delta_{t, t'} (E_{s, s'}-(-1)^{t+t'-s-s'}  E_{2n+1-s',2n+1-s} )+\delta_{s,s'}(E_{t',t}-(-1)^{t+t'-s-s'}  E_{2n+1-t,2n+1-t'})  \\
&&+\delta_{s'+t, 2n+1}((-1)^{t'-s'}E_{s, 2n+1-t'}-(-1)^{t-s} E_{t', 2n+1-s})
\end{eqnarray*}
Since $J=-\id$ on the simple component of $C_{n-k}$,
then \[
J(E_{t',t}-(-1)^{t+t'-s-s'}  E_{2n+1-t,2n+1-t'})=-(E_{t',t}-(-1)^{t+t'-s-s'}  E_{2n+1-t,2n+1-t'}). \]
This implies that
\begin{eqnarray*}
&& J([ [z, E_{s,t} +(-1)^{t-s-1} E_{2n+1-t,2n+1-s}], E_{t',s'}+(-1)^{t'-s'} E_{2n+1-s',2n+1-t'}  ]) \\
&=&(E_{s,t} +(-1)^{t-s-1} E_{2n+1-t,2n+1-s})(E_{t',s'}+(-1)^{t'-s'} E_{2n+1-s',2n+1-t'}) \\
&&+(E_{t',s'}+(-1)^{t'-s'} E_{2n+1-s',2n+1-t'}) (E_{s,t} +(-1)^{t-s-1} E_{2n+1-t,2n+1-s}).
\end{eqnarray*}
This implies that (\ref{scom}) holds for $X, Y\in \fgbb$.

\begin{theorem}
 Let $\si_k$ $(1<k<n)$ be the order 4 outer automorphism of  $A_{2n-1}$ represented by the marking $(c_k, d)$ defined in (\ref{mark}).
There exists a Lie superalgebra structure on $L= \fgaa+\fgbb$ with the Lie super bracket $[\cdot,\cdot]_s$ such that the even part is $\fgaa$ and the odd part is $\fgbb$.
Furthermore, we have $L\cong D(k,n-k)$.
\label{thm1}

\end{theorem}


\begin{proof}
Firstly, since $z\in\fgcc$, then  $[[z, \fgbb],\fgbb]\subset  \fgaa$ and hence $[\fgbb, \fgbb]_s\subset \fgaa$.
Let $X, Y\in \fgbb$, then
\[
[X, Y]_s=J([[z, X], Y])=J([[z, Y], X])=[Y, X]_s.
\]
Hence  $[\cdot, \cdot]_s: \fgbb\times \fgbb\longrightarrow \fgaa$ is a symmetric bilinear operator.
Also by direct computation, we get that
\begin{equation}
[\fgbb,\fgbb]_s=\fgaa.
\end{equation}

By the definition of Lie superalgebra, if we want to make $L$ to be a Lie superalgebra, we should verify that
\begin{equation}
\begin{array}{ll}
\mbox{(a)} & \mbox{ The map $S^2(\fgbb)\longrightarrow \fgaa$ induced by the symmetric operator} \\
&\mbox{  $[\cdot,\cdot]_s: \fgbb\times \fgbb\longrightarrow \fgaa$ is a $\gz$-homomorphism;}\\
\mbox{(b)} & \mbox{For $X,Y,Z\in\fgbb$, the super Jacobi identity holds, that is: } \\
&  [X, [Y, Z]_s]_s = [[X,Y]_s, Z]_s-[Y, [X, Z]_s]_s.
\end{array}
\label{super}
\end{equation}

Now we prove (\ref{super})(a). Let $U\in \fgaa$, $X,Y\in \fgbb$. Then
\begin{eqnarray*}
[U, [X,Y]_s]_s&=&[U, XY+YX]_s=U(XY+YX)-(XY+YX)U \\
&=&UXY+UYX-XYU-YXU,
\end{eqnarray*}
and
\begin{eqnarray*}
&&[[U,X]_s, Y]_s+[X,[U,Y]_s]_s \\
&=&[UX-XU, Y]_s+[X, UY-YU]_s \\
&=&(UX-XU)Y+Y(UX-XU)+X(UY-YU)+(UY-YU)X\\
&=&UXY+UYX-XYU-YXU.
\end{eqnarray*}
Then $[U, [X,Y]_s]_s=[[U,X]_s, Y]_s+[X,[U,Y]_s]_s$.
This prove (\ref{super})(a).

Next we prove (\ref{super})(b). Let $X,Y,Z\in \fgbb$. Then
\begin{eqnarray*}
[Z, [X,Y]_s]_s&=&[Z, XY+YX]_s=Z(XY+YX)-(XY+YX)Z  \\
&=&ZXY+ZYX-XYZ-YXZ,
\end{eqnarray*}
and
\begin{eqnarray*}
&& [[Z,X]_s, Y]_s-[X,[Z,Y]_s]_s  \\
&=&[ZX+XZ, Y]_s-[X, ZY+YZ]_s \\
&=&(ZX+XZ)Y-Y(ZX+XZ)-X(ZY+YZ)+(ZY+YZ)X\\
&=&ZXY+ZYX-XYZ-YXZ.
\end{eqnarray*}
Then $[Z, [X,Y]_s]_s=[[Z,X]_s, Y]_s-[X,[Z,Y]_s]_s$.
This prove (\ref{super})(b).

Now we get that $L$ is a Lie superalgebra. By Lemma \ref{irr}, we know that the action of $\fgaa$ on $\fgbb$ is irreducible,
then $L$ is simple.
Also by Lemma \ref{irr}, the $\fgaa$-representation on $\fgbb$ is isomorphic to $\bc^{2k}\otimes \bc^{2n-2k}$.
Then $L=\fgaa \oplus\fgbb$ is a simple Lie superalgebra for which the representation of $\fgaa $ on $\fgbb$ is the same as
the Lie superalgebras $D(m,n)$.
This implies $L\cong D(k,n-k)$ by \cite[Prop.2.1.4]{ka77}.
\end{proof}


Similarly, we have the following results.   
\begin{theorem}
Let $\si_1$ be the order 4 outer automorphism of  $A_{2n-1}$ represented by the marking $(c_1, d)$.
Then there exists a Lie superalgebra structure on $L=\fgaa+\fgbb$ with the Lie super bracket $[\cdot,\cdot]_s$ such that the even part is $\fgaa$ and the odd part is $\fgbb$.
Furthermore,   we have $L\cong C(n)$.
\label{thm2}
\end{theorem}

\begin{theorem}
Let $\si_k$ $(1\leq k\leq n)$ be the order 4 outer automorphism of  $A_{2n}$ represented by the marking $(c_k, d)$ defined in (\ref{mark}).
There exists a Lie superalgebra structure on $L=\fgaa+\fgbb$ with the Lie super bracket $[\cdot,\cdot]_s$ such that the even part is $\fgaa$ and the odd part is $\fgbb$.
Furthermore, we have $L\cong B(n-k,k)$.
\label{thm2.5}
\end{theorem}

\begin{theorem}
Let $\fg=D_n$ and $\si$ in Table \ref{table1}(e).
Then the map $[\cdot,\cdot]_s$ defines a Lie superalgebra structure on $L=\fgaa+\fgbb$.
Furthermore  $L\cong A(0, n-2)$.
\label{thm3}
\end{theorem}

\begin{remark}
Generally, if we define $[X, Y]_s=[[z, X], Y]$ for $X, Y\in \fgbb$, this does not give a Lie superalgebra structure on $L=\fgaa+\fgbb$.
For example, let $\fg=A_{2n-1}$ and consider $\sigma_k (1<k<n)$.
Choose $a=E_{1, k+1}+(-1)^{k-1}E_{2n-k, 2n}, b=E_{k+1, 1}+(-1)^k E_{2n,2n-k}, c=E_{k+2, 1}+(-1)^{k-1}E_{2n, 2n-k-1}\in \fgbb$.
Then we have
\begin{eqnarray*}
&& [a, [b, c]_s]-[[a, b]_s, c]+[b, [a,c]_s] \\
&=& [E_{1, k+1}+(-1)^{k-1}E_{2n-k, 2n}, [-E_{k+1, 1}+(-1)^k E_{2n,2n-k}, E_{k+2, 1}+(-1)^{k-1}E_{2n, 2n-k-1}]] \\
&&- [[E_{1, k+1}-(-1)^{k-1}E_{2n-k, 2n},  E_{k+1, 1}+(-1)^k E_{2n,2n-k}] , E_{k+2, 1}+(-1)^{k-1}E_{2n, 2n-k-1}] \\
&& +[ E_{k+1, 1}+(-1)^k E_{2n,2n-k} ,[E_{1, k+1}-(-1)^{k-1}E_{2n-k, 2n}, E_{k+2, 1}+(-1)^{k-1}E_{2n, 2n-k-1}]] \\
&=& -[E_{1,1}-E_{k+1,k+1}+E_{2n-k, 2n-k}-E_{2n, 2n}, E_{k+2, 1}+(-1)^{k-1}E_{2n, 2n-k-1}]  \\
&& +[ E_{k+1, 1}+(-1)^k E_{2n,2n-k} , -E_{k+2, k+1}-E_{2n-k, 2n-k-1}] \\
&=& 2 (E_{k+2, 1}-(-1)^kE_{2n, 2n-k-1})\neq 0.
\end{eqnarray*}

But for the case $\fg=A_{2n}$ and the automorphism $\si_n$, we have $\fgaa=C_n$, which is simple.
Then $J=\id$ and  $[X, Y]_s=[[z, X], Y]$ for $X, Y\in \fgbb$.
This gives a Lie superalgebra structure of type $B(0, n)$ on $L=\fgaa+\fgbb$.
\label{rem}
\end{remark}

\subsection{Exceptional cases}

In this part, we briefly introduce the proof of Theorem \ref{main} for two exceptional cases, which is  in Table \ref{table1}(f) and Table \ref{table1}(b) with $\fg = A_5$.

\underline{Case 1: Table \ref{table1}(f)} 

Let $\fg$ be the simple Lie algebra of type $E_6$. 
Let $\{\alpha_1, \dots, \alpha_6\}$ be the set of simple roots of $E_6$ in the Bourbaki labeling, where $\alpha_2$  is the simple root associated with the branch vertex of the Dynkin diagram.
Let $X_{\al_1}, X_{\al_2}, \ldots, X_{\al_6}$ be the root vectors attached to the simple roots $\al_1, \al_2, \ldots, \al_6$ respectively.
Let $\si$ be an order 4 automorphism of $E_6$ represented by
\[ \si X_{\al_3}=X_{\al_5}, \si X_{\al_5}=X_{\al_3}, \si X_{\al_2}=X_{\al_2}, \si X_{\al_4}=X_{\al_4},
  \si X_{\al_1}=i X_{\al_6}, \si X_{\al_6}=iX_{\al_1}. \]

The root system data of $E_6$ is well-known (i.e.  \cite{Bourbaki}).  
A straightforward but lengthy computation shows that  $\fgaa$ is isomorphic to $B_3 \oplus A_1$. Here the $B_3$ factor is generated by the $\sigma$-invariant combinations of roots involving $\alpha_2, \alpha_3+\alpha_5, \alpha_4$, and the $A_1$ factor corresponds to the highest root orbit.
The subspace $\fgbb$ is 16-dimensional and isomorphic to $\mathrm{spin}_7 \otimes \bc^2$ as a $\fgaa$-representation. 
Here $\mathrm{spin}_7$ is the 8-dimensional spinor representation of $\mathfrak{so}_7$ and $\bc^2$ is the standard representation of the $A_1$ factor.
%

We can write
the subalgebra $A_1$ of $\fgaa$ as $A_1=\mathrm{Span}\{ H_0, X_0, Y_0\}$.
Write $V=\mathrm{Span}\{v_1, v_2, \ldots, v_8\}$. 
By  multiplying  a nonzero coefficient,  we can choose $z \in \fgcc$  and assume that
\[ [[z, v_k], v_j]=\left\{\begin{array}{ll}
0,& \mbox{ if $k+j\neq 9$}, \\
X_0, & \mbox{ if $j+k=9$}.\end{array}\right.
\]
Let $W=\mathrm{Span}\{w_1, w_2, \ldots, w_8\}$ with $w_k=[Y_0, v_k]~ (1\leq k\leq 8)$.
Then $[[z, w_k], w_j]=0$ if $k+j\neq 9$ and $\fgbb=V\oplus W$.
Furthermore,  we have
\[\left\{ \begin{array}{ll}
{ [[z, v_k], v_j]=0, ~[[z, w_k], w_j]=0, ~[[z, v_k], w_j]\in B_3}  &\mbox{ if $k+j\neq 9$}, \\
{ [[z, v_k], v_j]=X_0}, ~ [[z, w_k], w_j]=-Y_0, &  \\
{ [[z, v_k], w_j]= X_{B_3}-\frac{1}{2}H_0\mbox{ with }X_{B_3}\in B_3}, &\mbox{ if $k+j=9$}.
 \end{array}\right.
\] 

Then the bilinear map defined in Theorem \ref{main} on $\fgbb$ are as follows:
\begin{equation}
 \label{sbra}
\left\{ \begin{array}{ll}
{ [v_k, v_j]_s=0, ~[ w_k , w_j]_s=0, ~[ v_k , w_j]_s=[[z, v_k], w_j] }  &\mbox{ if $k+j\neq 9$}, \\
{ [v_k, v_j]_s=-3X_0}, ~ [ w_k, w_j]_s=3 Y_0, ~ & \\
{ [  v_k, w_j]_s=[[z, v_k], w_j]+2H_0}, &\mbox{ if $k+j=9$}.
 \end{array}\right.
\end{equation}

The super Jacobi identity follows from a case-by-case verification using the commutation relations of $E_6$. 
Finally, we get the following result. 
\begin{theorem}
Let $\fg=E_6$ and $\si$ in Table \ref{table1}(f).
Then the map $[\cdot,\cdot]_s$ defines a Lie superalgebra structure on $L=\fgaa+\fgbb$.
Furthermore  $L\cong F(4)$.
\label{thm4}
\end{theorem}

\underline{Case 2: Table \ref{table1}(b) with $\fg = A_5$}

In this part, we first recall the construction of the exceptional Lie superalgebra $D(2,1; \alpha), ~ \alpha\neq 0, -1$,
and then prove Theorem \ref{main} for $(A_5, \si)$ in Table \ref{table1}(b).

The exceptional Lie superalgebras  $D(2,1; \alpha)$ with $\al\neq 0, -1$ form a one-parameter family
of superalgebras of rank 3 and dimension 17.
The even part is isomorphic to $\fs\fl(2, \bc)^3$ and the odd part is the spinor representation of the even part.
In this part, we recall the construction of  $D(2,1; \alpha)$ with $\al\neq 0, -1$.
More details can be found in \cite[pp.17--18]{sch}.

Choose three 2-dimensional vector space $V_i$, $i=1,2,3$.
Let $\phi_i$ be a non-degenerate skew-symmetric bilinear form on $V_i$
(any two such forms are proportional). It is well-known that $\fsp(\phi_i)\cong \fsp(V_i)$.
We define
\[
\Lz=\fsp(\phi_1)+\fsp(\phi_2)+\fsp(\phi_3),  \Lo=V_1\otimes V_2 \otimes V_3.
\]

Up to a factor, there exists just one $\fsp(\phi_i)$-invariant bilinear mapping:
\[
P_i: V_i\times V_i\longrightarrow \fsp(\phi_i).
\]
We choose the normalization such that
\begin{equation}
P_i(x_i,y_i)z_i=\phi_i(y_i,z_i)x_i-\phi_i(z_i,x_i)y_i
\label{pp1}
\end{equation}
for all $x_i,y_i,z_i\in V_i$.
Then the most general ansatz for a $\Lz$-invariant bilinear mapping
\[
P: \Lo\times \Lo\longrightarrow \Lz
\]
is given by
\begin{eqnarray}
\label{pp}
&& P(x_1\otimes x_2 \otimes x_3, y_1\otimes y_2 \otimes y_3) \nonumber  \\
&=& \al_1 \phi_2(x_2,y_2) \phi_3 (x_3,y_3) P_1(x_1,y_1) +
\al_2 \phi_1(x_1,y_1) \phi_3 (x_3,y_3) P_2(x_2,y_2) \\
&& +\al_3 \phi_1 (x_1,y_1) \phi_2(x_2,y_2) P_3(x_3,y_3), \nonumber
\end{eqnarray}
for all $x_i,y_i \in V_i, i=1,2,3$.
Here $\al_1,\al_2,\al_3$ are some arbitrary complex numbers.
The mapping $P$ is symmetric and it is easy to check that
$L$ is a Lie superalgebra if and only if
\[
\al_1+\al_2+\al_3=0.
\]

The obtained Lie superalgebra
$\Gamma(\al_1,\al_2,\al_3)$ depending on three nonzero parameters $\al_i, i=1,2,3$. For
$\si \in S_3$ and any nonzero $\lambda\in\bc$, we have
\[
\Gamma(\al_1,\al_2,\al_3)=\Gamma(\al_{\si (1)}, \al_{\si (2)}, \al_{\si (3)}) =\Gamma(\lambda\al_1,\lambda\al_2,\lambda\al_3). \]
 Thus, we have a one-parameter family of Lie superalgebras
$D(2,1;\alpha) := \Gamma(\alpha,1,-1-\alpha )$ that are simple for $\al\neq 0,-1$.

We have $\Gamma(\al,1,-1-\al)\cong\Gamma(1,\al,-1-\al)\cong \Gamma(1,\alpha^{-1},-\al^{-1}-1)$,
and also
$\Gamma(\al,1,-1-\al)\cong\Gamma(-1-\al, 1,\al)$, which imply
\[
D(2,1;\al)\cong D(2,1;\al^{-1})\cong D(2,1;-1-\al).\]
The maps $\al\mapsto \al^{-1}$ and $\al\mapsto -1-\al$ generate an action of $S_3$ on $\bc \backslash \{0,-1\}$.
We have $D(2,1;\al)\cong D(2,1;\beta)$  if and only if $\beta\in S_3 \cdot \al$. This gives additional
isomorphisms
\[ D(2,1;\al)\cong D(2,1;-(1+\al)^{-1}\al) \cong D(2,1;-1-\al^{-1})\cong D(2,1;-(1+\al)^{-1}). \]
Thus, any orbit of $(\bc\backslash\{0,-1\})/S_3$ consists of six points, except for the orbit
corresponding to the three points $\al = 1,-2,-1/2$, and the orbit corresponding to the
two points $\al=-\frac{1}{2}\pm\sqrt{-\frac{3}{4}}$.
Finally, note that the Lie superalgebra $D(2,1;\alpha)$ is isomorphic to $D(2,1)={\mathfrak{osp}}(4|2)$ when $\alpha=1, -2$ or $-\frac{1}{2}$.

Now we give a more detailed description of $D(2,1; \alpha)$.
For $j=1, 2, 3$, assume $\{e_j , f_j\}$ to be a basis of $V_j$ and the non-degenerate skew-symmetric bilinear form $\phi_j$ is defined by
\begin{equation}
\phi_j(e_j, e_j)=\phi_j (f_j, f_j)=0, ~~ \phi_j(e_j, f_j)=-\phi_j(f_j, e_j)=-1.
\label{phj}
\end{equation}
Since $\fs\fp(\phi_j)\cong \fs\fl(2, \bc)$  and $V_j$ is the 2-dimensional standard representation of $\fs\fp(\phi_j)$,
then we can choose $\{h_j, x_j , y_j\}$  to be a basis $\fs\fp(\phi_j) $ satisfying
\begin{equation}
\left\{ \begin{array}{l}
[h_j, x_j]=2x_j, ~ [h-j, y_j]=-2y_j, ~~ [x_j, y_j]=h_j;  \\
 h_j. f_j=-f_j, ~ h_j. e_j=e_j, ~~ x_j.f_j=e_j, ~~\\
 x_j. e_j=0, ~~ y_j. e_j=f_j, ~~ y_j. f_j=0.\end{array}\right.
 \label{bphj}
 \end{equation}

Consider $\fg=A_5$ and the order 4 automorphism $\si$ is represented in Table \ref{table1}(b).
Then $\si$ acts on the root vectors of simple roots of $\fg$ as follows:
\[
  E_{1,2}\stackrel{1}{\longleftrightarrow} E_{5,6}, ~ E_{2,3}\stackrel{i}{\longleftrightarrow} E_{4,5}, E_{3,4}\stackrel{1}{\longleftrightarrow} E_{3,4}. 
\]
 
Then we get: \\
\underline{a basis of $\fgaa$}:
\begin{eqnarray*}
 &X_1\stackrel{\triangle}{=}E_{1,2}+E_{5,6}, ~ H_1\stackrel{\triangle}{=}E_{1,1}-E_{2,2}+E_{5,5}-E_{6,6}, ~ Y_1\stackrel{\triangle}{=}E_{2,1}+E_{6,5}, ~ &  \\
  &X_2 \stackrel{\triangle}{=}E_{3,4}, ~ H_1\stackrel{\triangle}{=}E_{3,3}-E_{4,4}, ~ Y_2\stackrel{\triangle}{=}E_{4,3}, ~ &   \\
   &X_3\stackrel{\triangle}{=}E_{1,5}+E_{2,6}, ~ H_3\stackrel{\triangle}{=}E_{1,1}+E_{2,2}-E_{5,5}-E_{6,6}, ~ Y_3\stackrel{\triangle}{=}E_{5,1}+E_{6,2}, ~ &
\end{eqnarray*}
\underline{a basis of $\fgbb$}:
\begin{eqnarray*}
 &  v_1\stackrel{\triangle}{=}E_{4,1}+E_{6,3}, ~  v_2\stackrel{\triangle}{=}-E_{2,3}-E_{4,5}, ~ v_3\stackrel{\triangle}{=}E_{3,1}+E_{6,4}, ~ v_4\stackrel{\triangle}{=}E_{2,4}-E_{3,5}, ~& \\
 &  v_5\stackrel{\triangle}{=}-E_{4,2}-E_{5,3}, ~  v_6\stackrel{\triangle}{=}E_{4,6}-E_{1,3}, ~ v_7\stackrel{\triangle}{=}E_{5,4}-E_{3,2}, ~ v_8\stackrel{\triangle}{=}E_{1,4}+E_{3,6}. ~&
\end{eqnarray*}
Here $\fgaa=\bc\{X_1, H_1, Y_1\}\oplus\bc\{X_2, H_2, Y_2\}\oplus\bc\{X_3, H_3, Y_3\}\cong \fs\fl(2, \bc)^3$, with $\{X_i, H_i, Y_i\}$ $(i=1,2,3)$ standard triple of $\fs\fl(2, \bc)$.
And $\fgbb$ is an 8-dimensional irreducible representation of $\fgaa$.

Define a map $J$ on $\fgaa$ as follows: \[
J=\left\{ \begin{array}{ll}
2\alpha \cdot \id, &  \mbox{ on $A_1=\bc\{X_1, H_1, Y_1\}$}, \\
1\cdot  \id, &  \mbox{ on $A_1=\bc\{X_2, H_2, Y_2\}$},\\
2 \beta\cdot   \id, &  \mbox{ on $A_1=\bc\{X_3, H_3, Y_3\}$},  \end{array}
\right.
\]
where $\alpha, \beta\in \bc$.
An element in the center of $\fgaa+\fgbb$ is chosen to be \[
z=\frac{1}{3}(E_{1,1}+E_{2,2}-2E_{3,3}-2E_{4,4}+E_{5,5}+E_{6,6}).
\]

Recall that  $[v, w]_s=J([[z, v], w])$ for $v, w\in \fgbb$.
If $[\cdot, \cdot]_s$ gives a Lie superalgebra structure on $L=\fgaa+\fgbb$,
the Jacobi identity should holds.

Since \[
[v_2, [v_7, v_8]_s]=[v_2, 0]=0,
\]
and
\begin{eqnarray*}
&& [[v_2, v_7]_s, v_8]-[v_7, [v_2, v_8]_s]   \\
&=&[\frac{1}{2}(-2\alpha H_1+2\beta H_3)-H_2, E_{1,4}+E_{3,6}]-[E_{5,4}-E_{3,2}, -2\beta(E_{1,5}+E_{2,6})]  \\
&=& \frac{1}{2}(-2\alpha+2\beta-2)\cdot (E_{1,4}+E_{3,6})-2\beta\cdot (E_{1,4}+E_{3,6}) \\
&=&  (-\alpha-\beta-1)\cdot (E_{1,4}+E_{3,6}),
\end{eqnarray*}
then the Jacobi identity should holds for $v_2, v_7, v_8\in \fgbb$ only if $\alpha+\beta+1=0$.
Also when $\alpha=0$ or $\beta=0$, we get that $[\fgbb, \fgbb]_s=A_1^2\neq \fgaa$, then $L=\fgaa+\fgbb$
can not be a simple Lie superalgebra.
Then we get a necessary condition for $[\cdot, \cdot]_s$  to be a Lie superalgebra structure on $L=\fgaa+\fgbb$:
\begin{equation}
\alpha+\beta+1=0, ~ \alpha\neq 0, ~\beta\neq -2. 
\label{albe}
\end{equation}

Under this condition, a direct computation yields $L=\fgaa+\fgbb\cong D(2,1;\alpha)$, where the isomorphism is given by: 
\begin{eqnarray*}
& v_1 \mapsto f_1\otimes f_2\otimes f_3, ~~ v_2 \mapsto f_1\otimes f_2\otimes e_3, ~~ v_3 \mapsto f_1\otimes e_2\otimes f_3, ~~ v_4 \mapsto f_1\otimes e_2\otimes e_3, ~~ & \\
&v_5 \mapsto e_1\otimes f_2\otimes f_3, ~~ v_6 \mapsto e_1\otimes f_2\otimes e_3, ~~ v_7 \mapsto e_1\otimes e_2\otimes f_3, ~~ v_8 \mapsto e_1\otimes e_2\otimes e_3, ~~ & \\
& H_j \mapsto h_j, ~~ X_j \mapsto x_j, ~~ Y_j \mapsto y_j,  ~~ \mbox{ for } j=1,2,3. &
\end{eqnarray*}
Here $\{h_j, x_j, y_j, e_j, f_j| j=1,2,3\}$ is listed in (\ref{phj}) and (\ref{bphj}).  
Therefore we get the following result.

\begin{theorem}
For $(\fg, \si)$ in Table \ref{table1}(b) with $\fg=A_5$, let \[
J=\left\{ \begin{array}{ll}
2\alpha \cdot \id, &  \mbox{ on $A_1=\bc\{X_1, H_1, Y_1\}$}, \\
1\cdot  \id, &  \mbox{ on $A_1=\bc\{X_2, H_2, Y_2\}$},\\
-2(1+\alpha) \cdot   \id, &  \mbox{ on $A_1=\bc\{X_3, H_3, Y_3\}$},  \end{array}
\right.
\]
with $\alpha\neq 0, -1$.
Then the operator $[v, w]_s=J([[z, v], w])$ on $\fgbb$ induces a Lie superalgebra structure on $L=\fgaa+\fgbb$
with even part $\fgaa$ and odd part $\fgbb$.
Also $L\cong D(2,1; \alpha)$.

In particular, when $\displaystyle \alpha=-\frac{1}{2}$, we get $\displaystyle L\cong D(2,1; -\frac{1}{2})\cong D(2,1)$.
\label{thm5}
\end{theorem}

\bigskip

\noindent{Proof of Theorem \ref{main}}: Theorem \ref{main} is proved by Theorems \ref{thm1}, \ref{thm2}, \ref{thm2.5}, \ref{thm3}, \ref{thm4} and \ref{thm5}.
The results are summarized in Table \ref{table6}. 

\begin{table}[ht]\renewcommand{\arraystretch}{2.4}
\centering
\begin{tabular}{|c|c|c|c|c| }
  \hline
  $\fg$ & $\si$ & $\fgaa$ &  $\fgbb$  & Lie superalgebra \\
  \hline
{ $ A_{2n-1} $} &\includegraphics[width=3.5cm]{A4.png}  &  {$ C_{n-1}\oplus \bc
  $} &    $\bc^{2n-2}\oplus(\bc^{2n-2})^*$  & $C(n)$ \\
    \cline{2-5}
&\includegraphics[width=3.5cm]{A5.png}  &  { $ D_p\oplus C_{n-p}   $} &    $\bc^{2p}\otimes\bc^{2n-2p}$ & $D(p, n-p)$\\
 \hline
{ $ A_{2n} $}  &\includegraphics[width=3.5cm]{A7.png}  &   { $ C_p\oplus B_{n-p}  $} &    $\bc^{2n-2p+1}\otimes\bc^{2p} $ & $B(n-p, p)$\\
    \cline{2-5}
   &\includegraphics[width=3.5cm]{A8.png}  &  { $ C_n
 $}&  $\bc^{2n} $ & {$B(0,n)$ } \\
  \hline
 {$ D_{n}$}
&\includegraphics[width=3.5cm]{D2-3.png}  &  { $  A_{n-2}\oplus  \bc  $} &    $\bc^{n-1}\otimes \bc\oplus (\bc^{n-1})^*\otimes \bc^*$ &   $ A(n-2, 0)$ \\
  \hline
  { $ E_6  $}     &\includegraphics[width=3.5cm]{E6-22.png}  &  { $B_{3}\oplus A_{1} $} &  $\mathrm{spin}_7 \otimes \bc^2$  & $F(4)$  \\
  \hline
  {$A_5$}
  &\includegraphics[width=1.8cm]{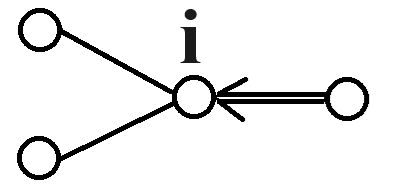}  &  { $  A_1^3  $} &    $\bc^{2}\otimes \bc^{2}\otimes\bc^{2} $ &   $ D(2, 1; \alpha)$ \\
  \hline
\end{tabular}
\caption{Generalized symmetric pairs  with symplectic subspace and Lie superalgebras}
\label{table6}
\end{table}
In Table \ref{table6}, we list the information about $\fgaa$ and $\fgbb$ attached to each generalized symmetric pair  with symplectic subspace $(\fg, \si)$. 
Here $\bc^n$ stands for the standard representation of the associated classical Lie algebra.
And $\mathrm{spin}_7$ stands for the 8 dimensional irreducible spinor representation of Lie algebra $B_3$. 
 $\hfill$$\Box$


\bigskip

If we consider $\fg$ to be semisimple, we have the following example.
\begin{remark}
\label{remAmn}
For $\fg=A_n\oplus A_n$, let $\si$ be an order 4 automorphism which permutes the two simple ideals of $\fg$.
Then $\sigma$ is outer,  $\fgaa=A_p\oplus A_q\oplus \bc$ and $\fgaa+\fgcc=A_p^2\oplus A_q^2\oplus \bc^2$ $(p+q=n-1)$.
Hence $\fgaa+\fgcc$ has a 1-dimensional center in $\fgcc$.
Similarly we can construct a Lie superalgebra structure on $\fgaa+\fgbb$.
And this Lie superalgebra is isomorphic to $A(p, q)$. 

The even part of the Lie superalgebra $G(3)$ contains a simple Lie algebra of type $G_2$, whose Dynkin diagram has a triple edge. In our theory, constructing $G(3)$ would require an outer automorphism with an orbit of three roots. It is well known that among complex simple Lie algebras, only $D_4$ admits an outer automorphism of order $3$. However, since $\dim D_4 = 28 < \dim G(3) = 31$, the Lie algebra $D_4$ is not large enough to accommodate $G(3)$ via our construction.

\end{remark}

\section{Dirac constants for Lie algebras and Lie superalgebras}

Let $(\fg, \si)$ be a generalized symmetric pair with symplectic subspace, and let $L=\fgaa+ \fgbb$ be the Lie superalgebra associated to this pair in the sense of Theorem \ref{main}. 
Then $\sigma^2$  is evidently an involution of Hermitian type of $\mathfrak{g}$. 
We study the relationship between two Dirac operators: 
the Dirac operator $D$ for the Lie algebra $\mathfrak{g} = \mathfrak{g}^{\sigma^2} + \mathfrak{g}^{-\sigma^2}$, 
and the Dirac operator $D_s$ for the associated Lie superalgebra $L = \fgaa+ \fgbb$.

Let $(\fg, \si)$ be a generalized symmetric pair with symplectic subspace and $z\in \fgcc$ be a nonzero element in the center of $\fgaa+\fgcc$.  
Together with $z$, the Killing form $\kapg(\cdot, \cdot)$ of $\fg$ induces the symplectic form on $\fgbb$ and $\fgdd$,  defined by 
\[
\langle X, Y\rangle =\kapg([z, X], Y), ~ \forall X, Y\in\fgbb.
\]

Choose $\{x_1, x_2, \ldots, x_s, \partial_1, \partial_2, \ldots, \partial_s\}$ to be a basis of $\fgbb$ satisfying $\langle x_i, \partial_j\rangle=\delta_{ij}$ 
and $\langle x_i , x_j\rangle=\langle \partial_i ,\partial_j\rangle=0$ for $1\leq i, j\leq s$.

\begin{lemma} 
With notation as above, then $\{-[z, \partial_1], -[z, \partial_2], \ldots, -[z, \partial_s], [z, x_1],   $ \\$ [z, x_2], \ldots, [z, x_s]\}$ is an orthonormal basis of $\fgdd$ with respect 
to the symplectic form $\langle \cdot, \cdot \rangle= \kapg([z, \cdot], \cdot)$.  
Also $\langle [z, x_i], -[z, \partial_j]\rangle=\delta_{ij}$ for any $1\leq i, j\leq s$. 
\end{lemma}
\begin{proof}
Since $\langle x_i, \partial_j\rangle =\delta_{ij}=\kapg([z, x_i], \partial_j)=-\langle \partial_j, x_i\rangle =-\kapg([z, \partial_j], x_i)$, 
then \[
\kapg(-[z, \partial_i], x_j)=\kapg([z, x_i], \partial_j)=\delta_{ij}.\] 
Also we have 
\[
\langle [z, x_i], -[z, \partial_j]\rangle=\kapg([z,[z, x_i]], -[z, \partial_j])=-\kapg( x_i,  [z, \partial_j])=-\kapg( [z, \partial_j], x_i)=\langle x_i, \partial_j\rangle.
\]
\end{proof}

The Dirac for the Lie algebra $\fg=\fg^{\si^2}\oplus \fg^{-\si^2}$ is 
\begin{eqnarray*}
D&=&\sum_{i=1}^s (-x_i\otimes [z, \partial_i]+ \partial_i \otimes [z, x_i]- [z, \partial_i]\otimes x_i + [z, x_i]\otimes\partial_i)  \\
&& \in \mathcal{U}(\fg)\otimes \mathcal{C}(\fgbb+\fgdd).
\end{eqnarray*}
where $\mathcal{C}(\fgbb+\fgdd)$ stands for the Weyl algebra of $\fg^{-\si^2}=\fgbb+\fgdd$ with respect to the Killing form $\kapg(\cdot, \cdot)$ on it.

\begin{proposition}{\cite[Prop.3.1.6]{HP06}}
The square of the Dirac operator $D$  for $\fg=\mathfrak{k}+\mathfrak{p}$ is given by 
\[
D^2=-\Omega_{\fg}\otimes 1+\Omega_{\mathfrak{k}_\Delta}+C,
\]
where $C$ is the constant $\|\rho_{\mathfrak{k}}\|^2-\|\rho_{\fg}\|^2$.
\label{HP1}
\end{proposition}

The Dirac operator for the Lie superalgebra $L=\fgaa+\fgbb$ is 
\[
D_s=\sum_{i=1}^s (x_i\otimes \partial_i-\partial_i\otimes x_i) \in \mathcal{U}(L)\otimes \mathcal{W}(\fgbb),
\]
where $ \mathcal{W}(\fgbb)$ stands for the Weyl algebra of $\fgbb$ with respect to the symplectic form $\langle \cdot, \cdot\rangle$ on it. 

\begin{proposition}[{\cite[Prop.2]{HP05}}]
The square of the Lie superalgebra Dirac operator $D_s$ for $L=\Lz\oplus \Lo$  is given by  
\[
D_s^2 =-\Omega_L\otimes 1 + \Omega_{\Lz \Delta} - C_s, 
\]
where $\Omega_L$ is the Casimir element of the Lie superalgebra $L$, $\Omega_{\Lz}$ is the Casimir element of the even part $\Lz$, and $C_s$ is a constant depending on the specific Lie superalgebra type. 
\label{HP2}
\end{proposition}

By Proposition \ref{HP1} and Table \ref{Weylvec}, we have the following results about the Dirac constant for Lie algebras.
\begin{theorem}
Let $(\fg, \si)$ be a generalized symmetric pair with symplectic subspace. Let $D$ be the dirac operator with respect to $\fg=\fg^{\si^2}+\fg^{-\si^2}$. Then the constant $C$ in $D^2$ is given in Table \ref{DcLie}.

 \begin{table}[ht] 
{\footnotesize
\renewcommand{\arraystretch}{1.4}
\begin{center}
\begin{tabular}{cccccccc} 
\hline
Case & Lie Algebra $\mathfrak{g}$ & Subalgebra $\mathfrak{g}^{\sigma^2}$ & $\|\rho_{\mathfrak{g}}\|^2$ &  $\|\rho_{\mathfrak{g}_{\sigma^2}}\|^2$ & $C=\|\rho_{\mathfrak{g}_{\sigma^2}}\|^2-\|\rho_{\mathfrak{g}}\|^2$  \\
\hline
(a) & $A_{2n-1}$ & $A_{2n-3} \oplus A_1 \oplus \mathbb{C}$ & $\frac{n(4n^2-1)}{6}$ &  $\frac{4n^3-12n^2+11n}{6}  $   & $-2n(n-1)$ \\
(b) & $A_{2n-1}$ & $A_{2p-1} \oplus A_{2q-1} \oplus \mathbb{C}$ & $\frac{n(4n^2-1)}{6}$ & $\frac{p(4p^2-1) + q(4q^2-1)}{6}$  & $-2npq$ \\
(c) & $A_{2n}$ & $A_{2p-1} \oplus A_{2q} \oplus \mathbb{C}$ & $\frac{n(2n+1)(2n+2)}{6}$ & $\frac{p(4p^2-1) + q(2q+1)(2q+2)}{6}$ & $-2npq-2pq-p^2-\frac{p}{2}$ \\
(d) & $A_{2n}$ & $A_{2n-1} \oplus \mathbb{C}$ & $\frac{n(2n+1)(2n+2)}{6}$ &$\frac{n(4n^2-1)}{6}$  &  $-\frac{n(2n+1)}{2}$\\
(e) & $D_n$ & $D_{n-1} \oplus \mathbb{C}$ & $\frac{n(n-1)(2n-1)}{6}$   & $\frac{(n-1)(n-2)(2n-3)}{6}$& $-(n-1)^2$  \\
(f) & $E_6$ & $D_5 \oplus \mathbb{C}$ & 78 &  30 &  $ -48$ \\
(g) & $A_5$ & $A_1^{\oplus 3} \oplus \mathbb{C}$ & 35/2 & 3/2 &  $-16$  \\
\hline
\end{tabular}\end{center} }
\caption{Dirac constant for generalized symmetric pairs}
\label{DcLie}
\end{table}
 \end{theorem}

Now we assume that Lie superalgebra $L$ has a non-degenerate Killing form.  
Recall the map $J$ on $\Lz$ is defined as in Theorem \ref{main}. 
We define $J^*(\Omega_{\Lz})$ as follow. 
Let $\{X_i\}$ be a basis of $\Lz$ with respect to the Killing form, and let $\{X_i^*\}$ be its dual basis. 
Then by definition, we have $\Omega_{\Lz}=\sum X_i\otimes X_i^*$. 
Define 
\begin{equation}
J^*(\Omega_{\Lz})=\sum X_i\otimes J^{-1}(X_i^*).
\end{equation} 

Actually, $J^*(\Omega_{\Lz})$ is a linear  combination of the Casimir elements of each simple component of $\Lz$.
By Schur's Lemma, it therefore acts as a constant on each irreducible representation of $\Lo$. Furthermore, when $L$ is a type 1 Lie superalgebra, 
its odd part $\Lo$ decomposes into a direct sum of two irreducible representations of $\Lz$. 
Since these two representations are dual to each other, the constant of  $J^*(\Omega_{\Lz})$  is uniquely determined on $\Lo$. 
For simplicity, we denote this constant also by  $J^*(\Omega_{\Lz})$. Next, we compute this constant and prove Theorem \ref{mainformula}. 

When $L$ is of type 2, since the $\Lz$-representation on $\Lo$ is irreducible, then we can write $\Lz=\fg\oplus \fh$ and $\Lo=U\otimes V$. 
Here $U$ and $V$ are irreducible representations of simple ideals $\fg$ and $\fh$, respectively. 
For $X, Y\in \fg$, we have 
\begin{eqnarray*}
\kapl(X, Y)&=&\mathrm{str}(\ad X \ad Y)=\mathrm{tr}(\ad X\ad Y)|_{\Lz}-\mathrm{tr}(\ad X\ad Y)|_{\Lo}  \\
&=&(1-l_U\cdot \dim  V)\cdot \kappa_{\fg}(X, Y) .
\end{eqnarray*}
Similarly, for $X, Y\in \fh$, we have 
 \[
\kapl(X, Y)=(1-l_V\dim U)\cdot \kappa_\fh(X, Y).\]
Then 
\begin{equation}
\Omega_{\Lz}=\frac{1}{1-l_U\dim V}\Omega_{\fg}+\frac{1}{1-l_V\dim U}\Omega_{\fh}.
\label{Omega1}
\end{equation}
Here $\Omega_{\fg}$ (resp. $\Omega_{\fh})$ is the Casimir element of $\fg$ (resp. $\fh$) with respect to $\kapg$ (resp. $\kappa_{\fh}$). 

For $L=B(0, n)$, since its even part $\Lz=B_n$ is simple,   we regard $\fh$ as redundant and $V=\bc$ as trivial. 
Thus, for this case, (\ref{Omega1}) reduces to the following formula: 
\begin{equation}
\Omega_{\Lz}=\frac{1}{1-l_U }\Omega_{\fg}.
\label{Omega2}
\end{equation}

When $L$ is of type 1, similarly we write $\Lz=\fg\oplus \fh\oplus \bc$ with one-dimensional center and $\Lo=U\otimes V\oplus U^*\otimes V^*$. 
Choose a nonzero element $z\in \bc$ such that $\ad z=1$ on $U\otimes V$ and $-1$ on $U^*\otimes V^*$. 
Then $\kapl(z, z)=-\frac{1}{2\dim U\dim V}$. 
Hence we get that 
\begin{equation}
\Omega_{\Lz}=\frac{1}{1-2l_U\dim V}\Omega_{\fg}+\frac{1}{1-2l_V\dim U}\Omega_{\fh}-\frac{1}{2\dim U\dim V}z\otimes z.
\label{Omega3}
\end{equation}

Specially, if $L=A(m,0)$ or $C(n)$, we regard $\fh$ as redundant and $V=\bc$ as trivial. 
Similarly, for this case, (\ref{Omega2}) reduces to the following formula: 
\begin{equation}
\Omega_{\Lz}=\frac{1}{1-2l_U }\Omega_{\fg}-\frac{1}{2\dim U}z\otimes z.
\label{Omega4}
\end{equation}

\bigskip 

\noindent{Proof of Theorem \ref{mainformula}: }
We illustrate two examples to prove this theorem. 
Let $L=B(m,n)$, then $\Lz=B_m\oplus C_n$ and $\Lo\cong \bc^{2m+1}\otimes \bc^{2n}$, where $\bc^{2m+1}$ (resp. $ \bc^{2n}$) stands for the standard representation of $B_m$ (resp. $C_n$).
By (\ref{Omega1}), we have  
\begin{eqnarray*}
J^*(\Omega_{\Lz})&=&\frac{2m-1}{2m-2n-1} {\Omega_{\fs\fo}}|_{L_{\bar{1}}}-\frac{2n+2}{2n-2m+1}  {\Omega_{\fs\fp}}|_{L_{\bar{1}}} \\
&=& \frac{2m-1}{2m-2n-1}\cdot \frac{m}{2m-1}-\frac{2n+2}{2n-2m+1}\cdot \frac{2n+1}{4(n+1)} \\
&=& \frac{2m+2n+1}{2(2m-2n-1)}.
\end{eqnarray*}
Here the minus appears in $J^*(\Omega_{\Lz})$ because $J$ acts by $-1$ on $C_n$ of $\Lz$, and  $\left.{\Omega_{\fs\fo}}\right|_{L_{\bar{1}}}= \frac{m}{2m-1}$ and $\left.{\Omega_{\fs\fp}}\right|_{L_{\bar{1}}}= \frac{2n+1}{4(n+1)}$ are got by Table \ref{Index}.

Next we consider the case of $L=A(n-2, 0)$. Then $\Lz=A_{n-2}\oplus \bc$, and $\Lo=\bc^{n-1}\oplus (\bc^{n-1})^*$.  
By Table \ref{Index}, we have $l_{\bc^{n-1}}=\frac{1}{2n-2}$.  
Recall that \[
J=\left\{\begin{array}{ll}
\id, & \mbox{ on }A_{n-2}, \\
-(n-2)\cdot \id, & \mbox{ on } \bc. \end{array}\right.
\]
Then by (\ref{Omega4}), we get 
\begin{eqnarray*}
J^*(\Omega_{\Lz})&=&\frac{n-1}{n-2} {\Omega_{\fs\fl}}|_{L_{\bar{1}}}-\frac{1}{n-2}\cdot (-\frac{1}{2(n-1)})\\
&=&\frac{n-1}{n-2}\cdot \frac{n(n-2)}{2{(n-1)}^2}+\frac{1}{2(n-1)(n-2)} \\
&=& \frac{n-1}{2(n-2)}.
\end{eqnarray*} 

After similar computation, we list all the data about $C, J^*(\Omega_{\Lz}), h^\vee_{\mathrm{sign}}$ and $\dim \Lo$ for each case  in Table \ref{JOmega},. 
Then Theorem \ref{mainformula} is then established by direct verification. 
\begin{table}[ht]
\centering
\begin{tabular}{| c|c|c|c|c|c|c|}
\hline
  &$\mathfrak{g}$ & $L$& $\dim \Lo$   & $J^*(\Omega_{\Lz})$ & $C$ (for $D^2$) & $h^\vee$ (for $L$) \\
\hline
(a) &$A_{2n-1}$ & $C(n)$ & $4(n-1)$ & $\frac{n}{2(n-1)}$ &  $ 2n(n-1)$ & $n-1$ \\
\hline
(b) &$A_{2n-1}$  & $D(p,n-p)$ & $4p(n-p)$ & $\frac{n}{2(p-q-1)}$ &  $ 2npq $ & $p-q-1$ \\
\hline
(c) &$A_{2n}$   & $B(n-p,p)$ & $4p(n-p)+2p$ & $\frac{2n+1}{2(2q-2p-1)}$ & $ 2npq+2pq+p^2+\frac{p}{2}$& $q-p-\frac{1}{2}$  \\
\hline
(d) &$A_{2n}$   & $B(0,n)$ & $2n$ & $\frac{1}{2}$ & $\frac{n(2n+1)}{2}$& $\frac{2n+1}{2} $  \\
\hline
(e) &$D_n$   & $A(n-2,0)$ & $2(n-1)$ & $\frac{n-1}{2(n-2)}$ & $(n-1)^2$& $n-2$  \\
\hline
(f) &$E_6$   & $F(4)$ & $16$ & $1$ & $48$ & $3$  \\
\hline
\end{tabular}
\caption{$J^*(\Omega_{\Lz})$  for Lie superalgebras}
\label{JOmega}
\end{table}
\hfill\qedsymbol 


\section{Kostant constant and superadmissibility index}

In \cite{Kos1} and \cite{Kos}, Kostant studies $\Lz+\Lo$, where $\Lz$ is a Lie algebra which admits a non-singular symmetric bilinear form, 
and $\Lo$ has a non-singular $\Lz$-invariant alternating bilinear form. 
He defines a $\nu^*$-map  by $\nu^*: \Lz \hookrightarrow \fsp(\Lo)\cong S^2(\Lo) \hookrightarrow \mathcal{W}(\Lo)$.
He also gives a necessary and sufficient condition for $\Lz+ \Lo$ to become a Lie superalgebra. 

\begin{proposition}[{\cite[Thms.0.1, 2.11]{Kos}}]
Given a Lie algebra $\Lz$ and its representation $\Lo$, where $\Lz$ admits a non-singular symmetric bilinear form and $\Lo$ admits a non-singular $\Lz$-invariant alternating bilinear form, then $L=\Lo+ \Lz$ is a Lie superalgebra of Riemannian type if and only if $\nu^*$ maps the Casimir element of $\Lz$ to a constant in $\mathcal{W}(\Lo)$. 
Moreover, $\nu^*(\Omega_{\Lz})=\frac{1}{8} \mathrm{tr}\Omega_{\Lz}|_{\Lo}$ is also equal to the constant in $D_s^2$ of the Dirac operator for Lie superalgebra.  
\label{Kos}
\end{proposition}
For convenience, we denote $\KKs=\left. \Omega_{\Lz}\right|_{\Lo}$ and call it the \emph{Kostant constant}.
In this section, we prove Theorem \ref{main12}. 

\subsection{Kostant constant}
  
When $L$ is of type 2,  we  write $\Lz=\fg\oplus \fh$ and $\Lo=U\otimes V$.  Let $C_U$ and $C_V$  represent the Casimir constants of $\fg$-representation $U$ and $\fh$-representation $V$, respectively.
Then by (\ref{Omega1}),  we have 
\begin{equation}
 \KKs=\frac{C_U}{1-l_U\dim V}+\frac{C_V}{1-l_V\dim U}.
 \label{SCas1}
 \end{equation} 
 
Specially, for $L=B(0, n)$,  we regard $\fh$ to be redundant and (\ref{SCas1}) turns to be the following degenerate form  
\[
\KKs=\frac{C_U}{1-l_U}.
\] 

When $L$ is of type 1,  we write $\Lz=\fg\oplus \fh\oplus \bc$  and $\Lo=U\otimes V\oplus U^*\otimes V^*$. 
Then by (\ref{Omega3}), we have 
\begin{equation}
\KKs= \frac{C_U}{1-2l_U\dim V}+\frac{C_V}{1-2l_V\dim U}-\frac{1}{2\dim U \dim V}.
 \label{SCas2}
 \end{equation} 
Specially, when $L=A(m, 0)$ or $C(n)$, we also regard $\fh$ to be redundant, and (\ref{SCas2}) turns to be  
 \[
 \KKs =\frac{C_U}{1-2l_U}-\frac{1}{2\dim U  }. \]
Using the information about Dynkin index and Casimir constant in Table \ref{Index}, we compute (\ref{SCas1}) and (\ref{SCas2})  as well as their degenerate forms  
separately, and then get the following Table \ref{SCas}. 

\begin{table}[H]
\centering 
\renewcommand{\arraystretch}{1.4}
\begin{tabular}{cccc}
\hline 
Type & $\Lz$ &  $\Lo$ &  Constant $\KKs$ \\
\hline
$A(m-1,n-1) $ & $A_{m-1}\oplus A_{n-1}\oplus \mathbb{C}$ & $ \bc^m\otimes (\bc^n)^*$ &  $\frac{\frac{(m^2-1)}{2m^2}}{1-2\cdot\frac{1}{2m}\cdot n}+\frac{\frac{n^2-1}{2 n^2}}{1-2\cdot\frac{1}{2n}\cdot m}-\frac{1}{2mn}$  \\ [0.5em] 
$(m\neq n)$ & & $\oplus (\bc^m)^*\otimes \bc^n$  & \\
\hline
$B(m,n) $ & $B_m+C_n$ & $\bc^{2m+1}\otimes \bc^{2n}$ & $\frac{\frac{m}{2m-1}}{1-\frac{1}{2m-1}\cdot 2n}+\frac{\frac{2n+1}{4(n+1)}}{1-\frac{1}{2(n+1)}\cdot(2m+1)}$ \\ [0.5em]
\hline
$B(0,n) $ & $C_n$ & $ \bc^{2n}$ & $\frac{\frac{2n+1}{4(n+1)}}{1-\frac{1}{2(n+1)}}$ \\ [0.5em]
\hline
$C(n) $ & $C_{n-1}\oplus \mathbb{C}$ & $ \bc^{2n-2}\oplus (\bc^{2n-2})^*$ & $\frac{\frac{2n-1}{4n}}{1-2\cdot\frac{1}{2n}}-\frac{1}{4(n-1)}$ \\ [0.5em]
\hline
$D(m,n) $ & $D_m+C_n$ & $ \bc^{2m}\otimes {\bc}^{2n}$ & $\frac{\frac{2m-1}{4(m-1)}}{1-\frac{1}{2(m-1)}\cdot (2n)}+\frac{\frac{2n+1}{4(n+1)}}{1-\frac{1}{2(n+1)}\cdot(2m)}$ \\ [0.5em]
$(m\neq n+1)$ & & & \\
\hline
$F(4)$ & $A_1\oplus B_3 $ & $\bc^2\otimes V_{\omega_3}$ & $\frac{\frac{3}{8}}{1-\frac{1}{4}\cdot 8}+\frac{\frac{21}{40}}{1-\frac{1}{5}\cdot 2}$ \\
\hline
 $G(3)$ & $A_1 \oplus G_2$ & $\bc^2\otimes W_{\omega_1}$ & $\frac{\frac{3}{8}}{1-\frac{1}{4}\cdot 7}+\frac{\frac{1}{2}}{1-\frac{1}{4}\cdot 2}$ \\[0.5em]
 \hline
\end{tabular}
\caption{Kostant constant for Lie superalgebras}
\label{SCas}
\end{table} 

\bigskip 

\noindent{Proof of Theorem \ref{main12}}: 
By the last column of Table \ref{SCas}, we can verify the constant $\KKs=\frac{1}{2}$ directly. Hence this proves Theorem \ref{main12}. 
\hfill\qedsymbol
 
\bigskip 

\noindent{Proof of Corollary \ref{KCos}}: 
By combining Proposition \ref{Kos} and Theorem \ref{main12}, we have that  the constant in $D_s^2$ satisfies 
\[
C_s=\nu^*(\Omega_{\Lz})=\frac{1}{8}\cdot \KKs\cdot \dim \Lo=\frac{1}{16} \dim \Lo.
\]  
This proves this corollary. 
\hfill\qedsymbol

\bigskip 

\subsection{Superadmissibility index of type $I\!I$}

Let $\mathfrak{g}$ and $\mathfrak{h}$ be complex simple Lie algebras, with irreducible representations $U$ and $V$, respectively. 
By (\ref{SCas1}), it is reasonable to define    
\begin{equation}
\Est(\fg, U; \fh, V) = \frac{C_U}{1 - l_U \dim V} + \frac{C_V}{1 - l_V \dim U}.
\label{EEt}
\end{equation}
We call this \emph{superadmissibility index of type $I\!I$} for $(\fg, U)$ and $(\fh, V)$.
By Theorem \ref{main2}, we know that  $(\fg, U; \fh, V)$ in (\ref{Ehalf}) satisfies the equation $\Est(\fg, U; \fh, V) =\frac{1}{2}$. 
\begin{equation}
\left\{ \begin{array}{ll} 
\mathrm{(1)}  &  (\fs\fo(2m+1), \bc^{2m+1}; \fs\fp(2n), \bc^{2n}), \\
\mathrm{(2)}  &  (\fs\fo(2m), \bc^{2m}; \fs\fp(2n), \bc^{2n}),  ~ (m\neq n+1)\\
\mathrm{(3)}  &  (\fs\fl(2), \bc^{2}; G_2, V_{\omega_1}), \\
\mathrm{(4)}  &  (\fs\fl(2), \bc^2; \fs\fo(7), V_{\omega_3}).
\end{array}
\right.
\label{Ehalf}
\end{equation}

Next we want to prove the following result. 
 \begin{theorem}
  (\ref{Ehalf}) exhausts all the pairs with $\Est=\frac12$.  
  \label{Es2}
 \end{theorem}

Recall that the irreducible representations of the simple Lie algebras with 
$0 < l_{U} < 1$   are listed in Table \ref{Index}.
   
Let $(\mathfrak{g}, U)$ be a non-trivial irreducible representation. Denote 
\[
\delta(U) = \frac{\dim\mathfrak{g}}{\dim U},\qquad r(U) = \frac{1}{l_U}.
\] 
By the relation between the Casimir eigenvalue and the Dynkin index, we have  
\[
C_U= \frac{l_U {\dim\mathfrak{g}}}{\dim U},
\]
and  
\[
\frac{C_U }{1-l_U\dim V} = \frac{l_U {\dim\mathfrak{g}}}{\dim U \cdot(1-l_U\dim V)}=\frac{ \delta(U)}{r(U)- \dim V}.
\]
 
Consequently, when both $U$ and $V$ are non-trivial, the equation $\Est = \frac{1}{2}$ is equivalent to
 \begin{equation}
 \frac{\delta(U)}{r(U)-\dim V} + \frac{ \delta(V)}{r(V)-\dim U} = \frac{1}{2}. 
 \label{EEE}
 \end{equation}
 
 The following lemma is well-known and useful. 
\begin{lemma}[\cite{AVE}]
Let $(\mathfrak{g}, W)$ be a non-trivial irreducible representation of a complex simple Lie algebra. If $l_W > 1$, then
  \[
\dim W > \dim\mathfrak{g}.
\] 
\label{lem}
\end{lemma}
 
 
\begin{proposition}
When $\Est=1/2$, then $l_U\leq 1$ and $l_V\le 1$. Moreover, $l_U=1$ or $l_V=1$ occurs only in the case $(\fg,U; \fh, V)=(A_1,\operatorname{adj}; A_1,\bc^2)$ or $(A_1,\operatorname{adj}; C_n,\bc^{2n})$. 
\label{prop1}
\end{proposition}
 
\begin{proof}
By symmetry, it suffices to rule out $l_U > 1$.  
In this case, the first term of (\ref{EEt}) is negative. Since $\Est= 1/2 > 0$, then the second term of (\ref{EEt}) must be positive, hence 
\[
 r(V)-N> 0 \quad\Longrightarrow\quad r(V) =\frac{1}{l_V}> N \geq 4.
\] 
Thus $V$ must appear in Table \ref{Index} with $0 < l_V < \frac{1}{4}$. Then we prove this theorem case by case. 
 
\underline{\noindent{Case 1:} $\fh=A_{m-1}, V=\bc^m$. } Then
\[
\dim V = m,\qquad r(V) =2m,\qquad \delta(V) = \frac{m^2-1}{m}.
\]
Since $r(V) > N$, we have $2m > N$. Then (\ref{EEE}) becomes 
\[
 \frac{\delta(U)}{r(U)-m} + \frac{ \frac{m^2-1}{m}}{2m-N}  = \frac{1}{2}.
\]
This implies 
\begin{equation}
 \frac{\delta(U)}{m-r(U)} =  \frac{ N-\frac{2}{m}}{2(2m-N)}.
\label{eqA1}
\end{equation}
Since $\delta(U) < 1$ and $r(U) < 1$, the left-hand side of (\ref{eqA1}) satisfies
\[
 \frac{\delta(U)}{m-r(U)}  < \frac{1}{m-1}.
\]
We claim that the right-hand side of (\ref{eqA1}) satisfies
\[
\frac{ N-\frac{2}{m}}{2(2m-N)}  > \frac{1}{m-1},
\]
which is equivalent to 
\[
(m-1)(N-\frac{2}{m})  > 2(2m-N).
\]
Since 
\[
(m-1)(N-\frac{2}{m}) -2(2m-N)=m(N-4)+N-2+\frac{2}{m}> 0,
\]
then $\frac{ N-\frac{2}{m}}{2(2m-N)}  > \frac{1}{m-1}$ as claimed. 
This contradicts (\ref{eqA1}).
 
\underline{\noindent\textbf{Case 2:} $\fh=C_m, V = \bc^{2m}$.} We have 
\[
\dim V = 2m,\qquad r(V) = 2m+2,\qquad \delta(V) = \frac{2m+1}{2}.
\]
Then (\ref{EEE}) becomes
\[
 \frac{\delta(U)}{r(U)-2m} + \frac{{2m+1}}{2(2m+2-N)}  = \frac{1}{2}.
\]
This implies 
\[
 \frac{N-1}{2(2m+2-N)}= \frac{\delta(U)}{2m-r(U)} < \frac{1}{2m-1}.
\]
But $ \frac{N-1}{2(2m+2-N)} > \frac{1}{2m-1}$, 
Which is equivalent to $2(N-3)m + N-3 > 0$.
Then this is a contradiction.
 
\underline{\noindent\textbf{Case 3:} $\fh=\fs\fo(p, \bc), V = \bc^{p}$.} In this case, we have  
\[
\dim V = p,\qquad r(V) = p-2,\qquad \delta(V) = \frac{p-1}{2}.
\]
Equation (\ref{EEE}) becomes
\[
 \frac{\delta(U)}{r(U)-p} + \frac{{p-1}}{2(p-2-N)}  = \frac{1}{2}.
\]
This implies  \[
\frac{N+1}{2(p-2-N)}= \frac{\delta(U)}{p-r(U)}  < \frac{1}{p-1}.
\]
Also since $\frac{N+1}{2(p-2-N)} > \frac{1}{p-1}$,
which is equivalent to $(N-1)p + N +3 > 0$, 
then this yields a contradiction. 
 
\underline{\noindent\textbf{Case 4:} $\fh=B_3, V=\mathrm{spin}_7$.} This is the $8$-dimensional spin representation of $\fh$, with data 
\[
\dim V = 8,\qquad r(V) = 5,\qquad \delta(V) = \frac{8}{21}.
\]
Since $r(V) > N \geq 4$, we must have $N = 4$. Equation (\ref{EEE}) becomes
 \[
\frac{\delta(U)}{r(U)-8} + \frac{21/8}{5-4}  = \frac{1}{2}.
\]
This implies 
\begin{equation}
 \frac{\delta(U)}{8-r(U)} = \frac{17}{8}.
\label{eqB31} 
\end{equation}
But since $\frac{\delta(U)}{8-r(U)} < \frac{1}{7}$, 
then this  yields a contradiction.  
Therefore $l_U > 1$ is impossible. 
 
If $l_U=1$, which means $U$ to be the adjoint representation of $\fg$, then \[
\Est=\frac{1}{1-\dim V}+\frac{C_V}{1-l_V \dim \fg}. \]
If $V$ is not trivial representation, then the condition $\Est=\frac{1}{2}$ forces $\frac{C_V}{1-l_V \dim \fg}>0$, and hence $0<l_V<1$. 
This implies that $V$ is in Table \ref{Index}. Next we prove that this occurs only in the case $(\fg,U; \fh, V)=(A_1,\operatorname{adj}; A_1,\bc^2)$ or $(A_1,\operatorname{adj}; C_n,\bc^{2n})$ by exhausting all $V$ in Table \ref{Index} case by case. 
We illustrate several examples to prove this. 

If $\fh=A_{n-1}$ and $V=\bc^n$, then $l_V=\frac{1}{2n}$, $C_V=\frac{n^2-1}{2n^2}$, and hence 
\[
\Est=\frac{1}{1-n}+\frac{n^2-1}{2n^2(1-\frac{1}{2n}\cdot \dim \fg)}=\frac{1}{2}.
\]
This implies $\dim \fg=4-\frac{2}{n}$. 
Since $\dim \fg\in \bn$, then this forces $n=2$. This implies the case $(\fg,U; \fh, V)=(A_1,\operatorname{adj}; A_1, \bc^2)$. 

Assume $\fh=C_{n}$ and $V=\bc^{2n}$. Then $l_V=\frac{1}{2(n+1)}$ and $C_V=\frac{2n+1}{4(n+1)}$. 
This implies \[
\Est=\frac{1}{1-2n}+\frac{2n+1}{4(n+1)(1-\frac{1}{2(n+1)}\dim \fg)}=\frac{1}{2}.
\]
After solving this equation, we get $\dim \fg=3$. This implies that $(\fg,U; \fh, V)=(A_1,\operatorname{adj}; C_n,\bc^{2n})$. 

Next we let $\fh=\fs\fo(p,\bc)$ and $V=\bc^p$. By Table \ref{Index}, we have $l_V=\frac{1}{p-2}$ and $C_V=\frac{p-1}{2(p-2)}$. 
Hence 
\[
\Est=\frac{1}{1-p}+\frac{\frac{p-1}{2(p-2)}}{1-\frac{1}{p-2}\dim \fg}=\frac{1}{2}.
\]
This implies $\dim \fg=\frac{p-3}{p+1}<1$. So for this case, we can not get $\Est=\frac{1}{2}$. 

Similarly, for all remaining cases of $V$ in Table \ref{Index}, computation shows that $\Est=\frac{1}{2}$ cannot hold.  Then  this proposition is proved.
\end{proof}
 
 
\bigskip

\noindent{Proof of Theorem \ref{Es2}: 
Proposition \ref{prop1} shows that if $(\fg,U; \fh, V)$ is $(A_1,\operatorname{adj}; A_1,\bc^2)$ or $(A_1,\operatorname{adj}; C_n,\bc^{2n})$, then $\Est = \frac{1}{2}$. Besides this, $(\fg, U)$ and $(\fh, V)$ should be both in Table \ref{Index}.  
After exhausting each pairs in Table \ref{Index}, we can compute that (\ref{Ehalf}) consists of all the solutions for this case. 
We also regard $(A_1,\operatorname{adj}; C_n,\bc^{2n})$ as the special case of $(B_m, \bc^{2m+1}; C_n, \bc^{2n})$ with $m=1$, and 
$ (A_1,\operatorname{adj}; A_1,\bc^2)$ as $(B_1, \bc^3; C_1, \bc^2)$. 
Then this proves that (\ref{Ehalf}) exhausts all the pairs with $\Est=\frac12$.  
\hfill\qedsymbol 

\bigskip 

\begin{remark}
In (\ref{EE}), if we assume $\fh=0$ and $V$ to be trivial, this is the degenerate form of $\Est$. 
In this case $\Est(\fg, U)=\frac{C_U}{1-l_U}$.  That $\Est>0$ forces $0\leq l_U<1$. 
By checking the representations in Table \ref{Index}, we find that only for $\fg=C_n$, there is 
$C_U=\frac{2n+1}{4(n+1)}$, $l_U=\frac{1}{2n+2}$ and $\Est=\frac{\frac{2n+1}{4(n+1)}}{1-\frac{1}{2n+2}}=\frac12$. 
This corresponds to the Lie superalgebra $B(0,n)$.
\label{Remd}
\end{remark}

\subsection{Superadmissibility index of type $I$}

Let $\fg$ be a complex simple Lie algebra with an irreducible representation $U$. 
Let $\fh$ be one with an irreducible representation $V$. 
By (\ref{SCas2}), it is reasonable to define    
\begin{equation}
\Eso(\fg, U; \fh, V) = \frac{C_U}{1 - 2l_U \dim V} + \frac{C_V}{1 - 2l_V \dim U}-\frac{1}{2\dim U \dim V}.
\label{EE}
\end{equation}
We call this \emph{superadmissibility index of type $I$} for $(\fg, U)$ and $(\fh, V)$.
In this part, we solve $\Eso(\fg, U; \fh, V) =\frac{1}{2}$.

\begin{theorem}
The solutions of $\Eso(\fg, U; \fh, V)=\frac{1}{2}$ are 
\[
(\fg, U; \fh, V)=(A_{m-1}, \bc^m; A_{n-1}, \bc^n), ~(m\neq n). 
\]
Similarly, if we assume $\fh=0$ and $V$-trivial, that is, we regard $\Eso(\fg, U)=\frac{C_U}{1-2l_U}-\frac{1}{2\dim U}$ as the degenerate form of $E_s^1$. Then its solutions are 
\[
(\fg,U)=(A_{m-1}, \bc^m) \mbox{ or }(C_{n}, \bc^{2n}).
\]
\label{Es1}
\end{theorem}
\begin{proof}
First, we verify that when $m\neq n$, $\Eso(A_{m-1}, \bc^m; A_{n-1}, \bc^n)=\frac{1}{2}$. 
By Table \ref{Index}, we have 
\[
l_U=\frac{1}{2m}, C_U=\frac{m^2-1}{2m^2}; l_V=\frac{1}{2n}, C_V=\frac{n^2-1}{2n^2}.
\]
Then 
\[
\Eso(\fg, U; \fh, V)=\frac{\frac{m^2-1}{2m^2}}{1-2\cdot \frac{1}{2m}\cdot n}+\frac{\frac{n^2-1}{2n^2}}{1-2\cdot \frac{1}{2n}\cdot m}-\frac{1}{2mn}=\frac{1}{2}. 
\]
Specially, $\Eso(A_{m-1}, \bc^m)=\frac{\frac{m^2-1}{2m^2}}{1-2\cdot \frac{1}{2m}}-\frac{1}{2m}=\frac{1}{2}$. 
Also, for $(\fg, U)=(C_{n}, \bc^{2n})$, we have 
\[
l_U=\frac{1}{2(n+1)}, ~~C_U=\frac{2n+1}{4(n+1)}. 
\] 
Hence $\Eso(C_{n}, \bc^n)=\frac{\frac{2n+1}{4(n+1)}}{1-2\cdot \frac{1}{2(n+1)}}-\frac{1}{4n}=\frac{1}{2}. $

Next we prove that these exhaust all the solutions. 
The proof is similar to the proof of Theorem \ref{Es2}.  

First, we claim that when $\Eso(\fg, U; \fh, V)=\frac{1}{2}$, the Dynkin index should satisfy that $l_V< \frac{1}{2}$ and $l_U< 1$ or vice versa. 
By the definition of $\Eso$, we know that at least one term in $\Eso$ should be positive. 
We might assume $ \frac{C_V}{1 - 2l_V \dim U}>0$, which implies $2l_V<1$, and hence $l_V<\frac{1}{2}$. 
Suppose $l_U>1$, then similar to the proof of Proposition \ref{prop1}, we get that $\frac{1}{l_V}>2\dim U$. 
Then $l_V<\frac{1}{6}$, that is, $V$ can only be the standard representation of the classical Lie algebras. 
We illustrate an example to prove that $l_U\geq 1$ can not hold. 
Assume $\fh=A_{m-1}$ and $V=\bc^m$, then $l_V=\frac{1}{2m}$ and $C_V=\frac{m^2-1}{2m^2}$. 
Then \[
\Eso=\frac{\frac{m^2-1}{2m^2}}{1-2\cdot \frac{1}{2m}\cdot \dim U}+\frac{C_U}{1-2l_U\cdot m}-\frac{1}{2m \dim U}=\frac{1}{2}.
\]  
This implies \begin{eqnarray*}
\frac{1}{2m-1}&\geq&\frac{\dim\fg}{\dim U\cdot (2m-\frac{1}{l_U})}=\frac{m^2-1}{2m(m-\dim U)}-\frac{1}{2m\dim U}-\frac{1}{2} \\
&>&\frac{m^2-1}{2m(m-\dim U)}-\frac{1}{2}>\frac{\dim U}{2(m-\dim U)}.
\end{eqnarray*}
This implies a contradiction and hence $l_U<1$. 
Similarly, for other classical simple Lie algebra $\fh$ and standard representation $V$, we can verify $l_U<1$. 
Then $\Eso(\fg, U; \fh, V)=\frac{1}{2}$ implies $l_V< \frac{1}{2}$ and $l_U< 1$ or vice versa as claimed. 
 
Finally, similar to the method in the proof of Theorem \ref{Es2}, we verify all the cases for $l_U< \frac{1}{2}$ and $l_V< 1$ in Table \ref{Index}, 
Therefore, we get that  $(\fg, U; \fh, V)=(A_{m-1}, \bc^m; A_{n-1}, \bc^n), ~(m\neq n)$ exhausts all the solutions for $\Eso=\frac{1}{2}$. 

When $\fh$ is redundant,  that $\Eso(\fg, U)=\frac{C_U}{1-2l_U}-\frac{1}{2\dim U}=\frac{1}{2}$ forces $l_U<\frac{1}{2}$. 
And similarly, we can prove that  $(\fg,U)=(A_{m-1}, \bc^m)$ or  $(C_{n}, \bc^{2n})$ consists all the solutions for $\Eso(\fg, U)=\frac{1}{2}$. 
\end{proof}

\bigskip

\noindent{Proof of Theorem \ref{main4}}:  
Theorem \ref{main4}(1)  is an immediate consequence of Theorem \ref{Es2} and Remark \ref{Remd}, whereas Theorem \ref{main4}(2) follows directly from  Theorem \ref{Es1}.
\hfill\qedsymbol

%
%

%

\end{document}